\documentclass[oneside,a4paper,11pt]{article}

\usepackage[
    backend=biber,
    style=numeric,
    sorting=none,
    giveninits=true,
    maxbibnames=99,
    date=year,
    doi=true,
]{biblatex}

\renewbibmacro*{in:}{}

\DeclareFieldFormat{title}{``#1''}

\DeclareFieldFormat{volume}{\mkbibbold{#1}}
\renewbibmacro*{volume+number+eid}{%
  \iffieldundef{volume}
    {}
    {\mkbibbold{\printfield{volume}}}%
  \iffieldundef{number}
    {}
    {\mkbibparens{\printfield{number}}}%
  \setunit{\addcomma\space}%
  \printfield{eid}}

\DeclareFieldFormat{pages}{#1}

\renewbibmacro*{date}{%
    \printdate%
    \setunit{\addspace}%
}

\AtEveryBibitem{%
    \clearlist{language}%
    \clearfield{issn}%
    \clearfield{isbn}%
    \clearfield{month}%
}

\renewbibmacro*{doi}{%
    \setunit{\addspace}%
    \printfield{doi}%
}

\DeclareFieldFormat{edition}{#1\addspace ed.,}

\renewbibmacro*{issue+date}{%
  \ifentrytype{article}
    {} 
    {%
      \iffieldundef{issue}
        {}
        {\printfield{issue}\setunit*{\addspace}}%
      \usebibmacro{date}%
    }%
}

\renewbibmacro*{note+pages}{%
  \ifentrytype{article}
    {%
      \printfield{pages}%
      \iffieldundef{year}
        {}
        {\setunit{\addspace}\printtext{(}\printfield{year}\printtext{)}}%
      \iffieldundef{note}
        {}
        {\setunit{\addcomma\addspace}\printfield{note}\newunit}%
    }
    {\printfield{note}\newunit\printfield{pages}} 
}

\usepackage{graphicx}
\usepackage[usenames,dvipsnames]{xcolor}
\usepackage{amsmath,amssymb,amsthm,bm}
\usepackage{authblk}
\usepackage{url}
\usepackage{accents}
\usepackage{textgreek}
\usepackage{labelfig}

\usepackage[
    textheight=216mm,
    textwidth=145mm,
    voffset=0mm,
    hoffset=0mm
]{geometry}

\usepackage[colorlinks=true, linkcolor=blue, citecolor=blue, urlcolor=blue]{hyperref}

\newtheorem{theorem}{Theorem}[section]
\newtheorem{lemma}[theorem]{Lemma}
\newtheorem{proposition}[theorem]{Proposition}
\newtheorem{corollary}[theorem]{Corollary}

\theoremstyle{thmstyletwo}%
\newtheorem{remark}[theorem]{Remark}

\theoremstyle{thmstylethree}%

\numberwithin{equation}{section}

\makeatletter
\def\Authorfont{\reset@font\fontsize{12bp}{14.5bp}\selectfont\titraggedcenter}%
\makeatother

\newcommand{\pd}[2][]{
\if\relax#1\relax
 \partial_{#2}%
\else
 \partial_{#2}^{#1}%
\fi}

\def\D{\mathrm{d}}
\def\gamm{\gamma}
\def\dcrit{d_{\textrm{c}}}
\def\ds{d_{\textrm{s}}}
\def\taust{\tau_{*}}
\def\Disp{\sigma}
\def\DispOld{\mathring{\sigma}}
\def\dUd{\kappa}
\def\mddU{a}
\def\nmo{n_{-}}
\def\DN{\mathcal{D}_{\mathcal{N}}}
\def\Obig{\mathcal{O}}

\begin{document}

\baselineskip=4.4mm

\date{}

\newcommand{\emailnote}[1]{\textsuperscript{*}\footnotetext[1]{#1}}

\title{On the exchange of stability for the subcritical laminar flow}

\author[1,2]{Vladimir Kozlov}
\author[3]{Oleg Motygin}


\affil[1]{Department of Mathematics, Link\"oping University, Link\"oping, SE-581 83, Sweden}

\affil[2]{Department of Mathematics and Computer Sciences, St Petersburg State University, St~Petersburg, 199034, Russia}

\affil[3]{Institute for Problems in Mechanical Engineering, Russian Academy of Sciences, St~Petersburg, 199178, Russia\protect\\[2mm]
{\small{}E-mail: \protect\href{mailto:vladimir.kozlov@liu.se}{vladimir.kozlov@liu.se}, \protect\href{mailto:mov@ipme.ru}{mov@ipme.ru}, \protect\href{mailto:o.v.motygin@gmail.com}{o.v.motygin@gmail.com}}}

\maketitle

\vspace{-12mm}

\begin{abstract}
We consider steady water waves in a two-dimensional channel bounded below by a flat, rigid bottom and above by a free surface. Surface tension is neglected, and the flow is rotational with constant vorticity $a$. We analyze an analytic branch of Stokes waves bifurcating from a subcritical laminar flow, with the wave period serving as the bifurcation parameter. Along this branch, the first eigenvalue of the Fr\'{e}chet derivative remains negative. Our main focus is the second eigenvalue; its sign plays a crucial role in the analysis of subharmonic bifurcations. This small eigenvalue determines the validity of the principle of exchange of stabilities: a positive sign confirms it, while a negative sign indicates its violation. Furthermore, a positive second eigenvalue corresponds to an increasing period along the bifurcation curve near the critical point, whereas a negative sign implies period decrease. We investigate how the sign of the second eigenvalue depends on the Bernoulli constant $R$ (equivalently, the laminar flow depth $d$) and the vorticity $a$. We show that for each $a$ there exists a critical depth $d_0(a)$ such that the second eigenvalue is positive for $d<d_0(a)$ and negative for $d>d_0(a)$. In the laminar flow, a stagnation point forms when the depth exceeds a threshold $\ds(a)$. We demonstrate that $d_0(a) < \ds(a)$ for $a > a_0 \approx -1.01803$, whereas $d_0(a) > \ds(a)$ for $a < a_0$.  We also verify the property of formal stability by a description of the domain in $(a,d)$ variables, where this property holds. Numerical illustrations  of these properties are presented in the paper.

{\small

\vspace{2mm}


}

\end{abstract}

\section{Introduction}

We consider steady surface waves in a two-dimensional channel bounded below by a flat,
rigid bottom and above by a free surface that does not touch the bottom. The surface tension is neglected and the water motion can be rotational.
In appropriate Cartesian coordinates $(X,Y )$, the bottom coincides with the line
$Y=0$ and gravity acts in the negative $Y$-direction. We choose the frame of reference so that the velocity field is time-independent as well as the free-surface, which is located in the half-plane $Y>0$.   We introduce the water domain
$$
\mathcal{D}=\mathcal{D}_\xi=\{(X,Y\,:\, X\in\mathbb{R},\;\;0<Y<\xi(X) \}
$$
and the free surface
$$
\mathcal{S}=\mathcal{S}_\xi=\{(X,Y)\,:\,X\in \mathbb{R},\;\;Y=\xi(X)\}.
$$
The function $\xi$  is  even and periodic. We denote the period by $\Lambda$.

 To describe the flow inside $\mathcal{D}$, we use a stream function $\Psi$. Then the velocity vector is given by  $(-\Psi_Y,\Psi_X)$. Since the surface tension is neglected, the function $\Psi$  after a certain scaling, satisfies the following free-boundary problem (see for example \cite{CSst} and \cite{KN14}):
\begin{equation}\label{K2a}
\begin{aligned}
&\Delta \Psi+\omega(\Psi)=0\quad\mbox{in\!}\quad\mathcal{D},\\
&\frac{1}{2}|\nabla\Psi|^2+\xi=R\quad\mbox{on\!}\quad\mathcal{S},\\
&\Psi=1\quad\mbox{on\!}\quad\mathcal{S},\\
&\Psi=0\quad\mbox{for\!}\quad Y=0,
\end{aligned}
\end{equation}
where $\omega$ is a vorticity function, $R$ is the Bernoulli constant.
We always assume that the function $\Psi$ is even and $\Lambda$-periodic with respect to $X$ as well as $\xi$. Since we will use the period as a parameter of the problem it is convenient to make the following change of the variable
$$
x=\lambda X,\quad y=Y\quad\mbox{and}\quad\lambda=\frac{\Lambda_0}{\Lambda},\quad \psi(x,y)=\Psi\Bigl(\frac{x}{\lambda},y\Bigr),\quad\eta(x)=\xi\Bigl(\frac{x}{\lambda}\Bigr).
$$
Here $\Lambda_0$ is a fixed period which will be chosen later.
The problem (\ref{K2a}) becomes
\begin{equation}\label{Okt31a}
\begin{aligned}
&(\lambda^2\partial_x^2+\partial_y^2)\psi+\omega(\psi)=0\quad\mbox{in\!}\quad D_\eta,\\[1.5mm]
&\psi(x,0)=0,\quad\psi(x,\eta(x))=1\quad\mbox{for\!}\quad x\in\mathbb{R},\\
&\frac{1}{2}\bigl(|\partial_y\psi|^2+\lambda^2|\partial_x\psi|^2\bigr)+\eta=R\quad\mbox{on\!}\quad S_\eta,
\end{aligned}
\end{equation}
where $ D_\eta=\{(x,y)\,:\, x\in\mathbb{R},\;0<y<\eta(x) \}$ and $S_\eta=\{(x,y)\,:\,x\in \mathbb{R},\;y=\eta(x)\}$.
Now the functions $\psi$ and $\eta$ have the same period $\Lambda_0$.

To construct and study bifurcations of solutions to the problem (\ref{Okt31a}), an important role is played by the Fr\'{e}chet derivative of the corresponding nonlinear problem.  As this derivative requires a fixed domain, we first flatten the fluid domain using an appropriate change of variables and compute the derivative in the transformed coordinates. The corresponding spectral problem in $(x,y)$-coordinates reads (see Sect.~\ref{SecJan16a})
\begin{equation}\label{Au5a}
\begin{aligned}
& (\lambda^2\partial_x^2+\partial_y^2+\omega'(\psi))v=0\quad\mbox{in\!}\quad D_\eta,\\
&v=0\quad\mbox{for\!}\quad y=0,
\end{aligned}
\end{equation}
and
\begin{equation}\label{Dec27a}
\lambda^2\psi_xv_x+\psi_yv_y-\frac{\widehat{\rho}}{\psi_y} v=\frac{\mu}{\psi_y} v\quad\mbox{on\!}\quad S_\eta,
\end{equation}
where
\begin{equation}
\label{eq:rhohatdef}
\widehat{\rho}=1+\lambda^2\psi_{xy}\psi_x+\psi_{yy}\psi_y.
\end{equation}
All coefficients in this spectral problem are well defined if there are no stagnation points on the free surface. The spectrum of the above spectral problem  is bounded from below and consists of eigenvalues of finite multiplicity. These spectral properties follow from the standard theory of compact self-adjoint operators for elliptic boundary value problems. The self-adjointness of the associated Dirichlet--Neumann operator is explicitly demonstrated in Sect.~\ref{SFeb7c} via the Green formula, while compactness arises from elliptic regularity and Sobolev embeddings on the periodic cell. For a standard reference, see, e.g.\ \cite[\S\,8.12]{GilbargTrudinger1983}.

In the paper we consider mostly the case $\omega=a$, where $a$ is a constant. We refer to the papers \cite{CVar} and \cite{CSrVar}, where an explanation of importance of this problem can be found together with many references. Usually in the bifurcation analysis of such problems it is assumed that the period is fixed. We assume that the Bernoulli constant is fixed and consider the period as a parameter of the problem (see \cite{KN11a}). According to \cite{KN14, Var23}, there exists a branch of Stokes waves $\Xi(t)=(\psi(x,y;t),\eta(x;t),\Lambda(t))$, $t\geq 0$,  starting from a subcritical laminar flow $(U,d)$, where $U$ depends only on the vertical variable $y$ and the constant $d>\dcrit$, with $\dcrit$ being the value at which the function
$$
\mathcal{R}(d)=\frac{1}{2}\Bigl(\frac{1}{d^2}-a+\frac{a^2d^2}{4}\Bigr)+d
$$
attains its minimum, and $R=\mathcal{R}(d)$. The branch is analytically parameterized by the parameter $t$. For our results it is sufficient to consider this branch only for small $t$.

Due to subcriticality, there is a frequency $\taust$, or the period $\Lambda_*=2\pi/\taust$ (we take then $\Lambda_0=\Lambda_*$), at which the bifurcation occurs.  It is shown in Sect.~\ref{SecJan16b} that the first eigenvalue of the Fr\'{e}chet derivative is always negative for all $t$. Denote by $\mu(t)$  the second  eigenvalue of the Fr\'{e}chet derivative  at the point $(\psi(t),\eta(t),\Lambda(t))$ (we will omit sometime other variables if they do not play role in the exposition). The functions $\Lambda(t)$ and $\mu(t)$ can be represented as follows:
$$
\mu(t)=\mu_2t^2+\Obig(t^4),\quad\Lambda(t)=\Lambda_*+\Lambda_2t^2+\Obig(t^4)
$$
and
$$
\lambda(t):=\frac{\Lambda_*}{\Lambda(t)}=1+\lambda_2t^2+\Obig(t^4).
$$
We prove that always
\begin{equation*}
\mu_2=-A\lambda_2,\quad A=U'(d)^2\taust H(\taust d),
\end{equation*}
where $H(z)$ is positive for $z>0$, see \eqref{Dec6ca} in Sect.~\ref{sect:mu2}. The validity of the principle of the exchange of stability at the bifurcation point $\taust$ means that $\mu_2>0$ (see \cite{CrRab,Ki1}).

\begin{figure}[t!]
\centering\vspace{1.25mm}
 \SetLabels
 \L (-0.03*0.8) $d$\\
 \L (0.41*0.35) $\dcrit$\\
 \L (0.375*0.48) $\ds$\\
 \L (0.62*0.48) $\ds$\\
 \L (0.36*0.6) $d_0$\\
 \L (0.86*-0.02) $a$\\
 \L (0.44*0.0) $a_0$\\
 \L (0.18*0.34) $M_+$\\
 \endSetLabels
 \leavevmode\AffixLabels{\includegraphics[width=80mm]{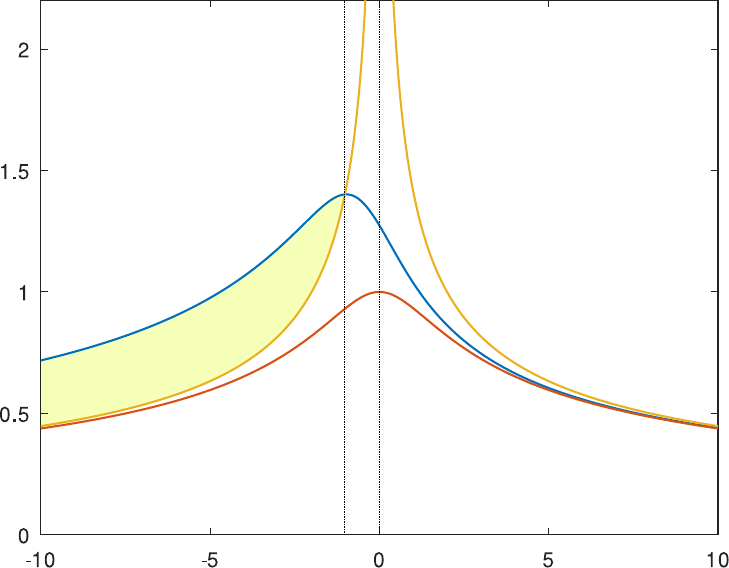}}\vspace{1mm}
 \caption{Positivity of $\mu_2(a,d)$ when the laminar flow $U$ is not unidirectional. Here $\dcrit(a)$ is the critical value of $d$, $\ds(a)=\sqrt{2/|a|}$ the depth when the laminar flow has a stagnation point on the surface or at the bottom, and $d_0(a)$ is the root of $\mu_2(a,d_0)=0$.
 The curves $\ds(a)$ and $d_0(a)$ intersect at $a=a_0\approx-1.01803$.}
\label{fig:mu2_pos_cf}
\end{figure}
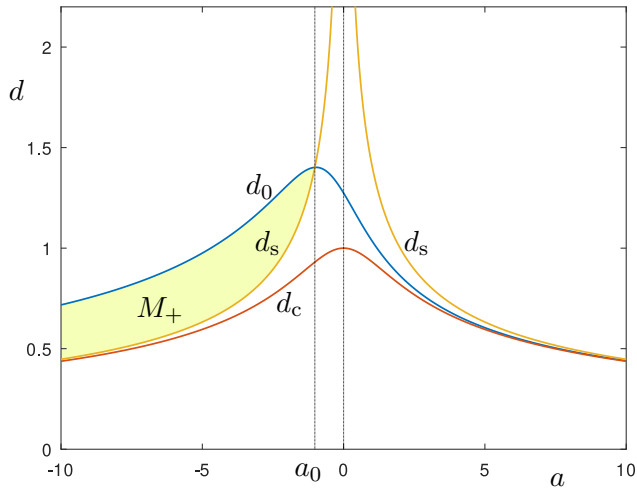

In Sect.~\ref{SJa22b},~\ref{sect:mu2} we analyze the sign of $\mu_2$ as a function of $a$ and $d$. We prove that $\mu_2(d,a)$ is positive for $d$ near $\dcrit(a)$ and negative for large $d$, ensuring that the equation $\mu_2(a,d)=0$ has at least one root for every fixed $a$. Our numerical results support the conjecture that this root, $d_0(a)$, is unique, with $\mu_2(a,d)>0$ for $d\in(\dcrit(a),d_0(a))$ and $\mu_2(a,d)<0$ for $d>d_0(a)$.

Of particular interest to us is the sign of $\mu_2$ in the case of counter-current flow. Introduce the quantity $\ds(a)=\sqrt{2/|a|}$, which is the depth of the laminar flow with a stagnation point at $y=d$ for $a>0$ and at $y=0$ for $a<0$. As stated in Sec.~\ref{SAu13a}, for $d>\ds(a)$, the laminar flow $(U,d)$ has counter-current flow near the bottom if $a<0$, and near the surface if $a>0$. In the latter case $\mu_2(a,d)<0$ for all $d>\ds$. Meanwhile, for $a<0$ there is a domain $(a,d)$ such that $a<a_0\approx-1.01803$ and $d_0(a)<d<\ds(a)$ where $\mu_2>0$. This domain $M_+$, where $\mu_2$ is positive for laminar flows with  counter-currents, is shown in Fig.~\ref{fig:mu2_pos_cf}.

The sign of $\mu_2$ plays important role in study of subharmonic bifurcations of Stokes waves, see \cite{Koz1a}. It was proved there that if $\mu_2(t)$ is positive for small $t$ along a bifurcation curve and then changes sign then subharmonic bifurcations occur near such point. The water waves appearing due to subharmonic bifurcations were one of the first examples of waves different from the Stokes and solitary waves. Their existence in irrotational case was proved in \cite{BDT1} and \cite{BDT2}.
In the case $a=0$ the sign of $\mu_2$ was studied in \cite{Koz23}. In \cite{Koz26} it was proved the following ``stability'' property: if $\mu_2>0$ then there are no even, periodic solutions with a multiple period in a small neighborhood of the Stokes wave considered in the space of solutions with the multiple period except the Stokes solution.

Important application of the positivity of the  sign of $\mu_2$ can be found in the study of stability of water waves with fixed period, see \cite{CSst2} and \cite{DLZ}. Different types of stability was discussed in the introduction to \cite{CSst2}, in particular Benjamin and Feir instability, which was analysed in \cite{BMV23}. In Sect.~\ref{SFeb7c} we present an application for study of the formal stability introduced in \cite[Sect.~4]{CSst2}. We show that formal stability holds in a certain domain of the $(a,d)$-plane and provide a description of this domain.

We also mention the work \cite{CKSch15}, where second and third order asymptotic expansions were derived for the wave profile and internal flow characteristics. While their analysis focuses on the approximation of the solution itself, we utilize asymptotic expansions to study the stability properties and the sign of the second eigenvalue of the Fr\'{e}chet derivative.


\section{Fr\'{e}chet derivative}\label{SecJan16a}

In this section we define the Fr\'{e}chet derivative of the operator corresponding to the problem (\ref{K2a})
in a fixed domain, obtained after application the flattering change of variables. After that, we evaluate the Fr\'{e}chet derivative in $(x,y)$ variables, which leads to the spectral problem (\ref{Au5a}), (\ref{Dec27a}).

We  use the flattening change of variables
$$
\hat{x}=x,\quad\hat{y}=\frac{dy}{\eta(x)},
$$
to reduce the problem to a fix strip
$$
Q=\{(\hat{x},\hat{y})\,:\,\hat{x}\in \mathbb{R},\;0<\hat{y}<d\}.
$$
Since
$$
\partial_{x}=\partial_{\hat{x}}-\frac{\hat{y}\eta'}{\eta}\partial_{\hat{y}},\quad\partial_y=\frac{d}{\eta}\partial_{\hat{y}},
$$
where $'$ means $\partial_{\hat{x}}$,
the system (\ref{Okt31a}) takes the form
\begin{equation}\label{K2aa}
\begin{aligned}
&F(\hat{\psi},\eta):=\Bigl(\lambda^2\Bigl(\partial_{\hat{x}}-\frac{\hat{y}\eta'}{\eta}\partial_{\hat{y}}\Bigr)^2+\Bigl(\frac{d}{\eta}\partial_{\hat{y}}\Bigr)^2\Bigr) \hat{\psi}+\omega(\hat{\psi})=0\quad\mbox{in\!}\quad Q,\\
&G(\hat{\psi},\eta):=\frac{1}{2}\Bigl(\lambda^2\Bigl|\Bigl(\partial_{\hat{x}}-\frac{\hat{y}\eta'}{\eta}\partial_{\hat{y}}\Bigr)\hat{\psi}\Bigr|^2
+\Bigl|\frac{d}{\eta}\partial_{\hat{y}}\hat{\psi}\Bigr|^2\Bigr)+\eta-R=0\quad\mbox{for\!}\quad\hat{y}=d,\\
&H(\hat{\psi}):=(\hat{\psi}-1)|_{\hat{y}=d}=0,\\[2mm]
&\hat{\psi}=0\quad\mbox{for\!}\quad\hat{y}=0,
\end{aligned}
\end{equation}
where
$$
\hat{\psi}({\hat{x}},\hat{y})=\psi\Bigl({\hat{x}},\frac{\hat{y}\eta({\hat{x}})}{d}\Bigr)\quad\mbox{or}\quad\psi(x,y)=\hat{\psi}\Bigl(x,\frac{dy}{\eta}\Bigr).
$$
Then the problem (\ref{K2aa}) is equivalent to
$$
(F(\hat{\psi},\eta),G(\hat{\psi},\eta),H(\hat{\psi}))=0,
$$
which is defined on $\Lambda_0$-periodic, even functions from $C^{2,\alpha}(Q)\times C^{2,\alpha}(\mathbb{R})$  satisfying $\hat{\psi}(\hat{x},0)=0$ and $\eta>0$.

We calculate the Fr\'{e}chet derivative at $(\hat{\psi},\eta)$:
\begin{multline}\label{Ju28a}
\mathcal{F}(u,\zeta):=\partial_{t}F(\hat{\psi}+tu,\eta+t\zeta)|_{t=0}=\Bigl(\lambda^2\Bigl(\partial_{\hat{x}}-\frac{\hat{y}\eta'}{\eta}\partial_{\hat{y}}\Bigr)^{\!2}
+\Bigl(\frac{d}{\eta}\partial_{\hat{y}}\Bigr)^{\!2}\Bigr)u+\omega'(\hat{\psi})u\\
{}-\lambda^2\Bigl(\frac{\zeta}{\eta}\Bigr)'\Bigl(\partial_{\hat{x}}-\frac{\eta'}{\eta}{\hat{y}}\partial_{\hat{y}}\Bigr)\hat{y}\partial_{\hat{y}}\hat{\psi}
-\lambda^2\Bigl(\partial_{\hat{x}}-\frac{\hat{y}\eta'}{\eta}\partial_{\hat{y}}\Bigr)\Bigl(\frac{\zeta}{\eta}\Bigr)'\hat{y}\partial_{\hat{y}}\hat{\psi}
-2\frac{d^2\zeta}{\eta^3}\partial_{\hat{y}}^2\hat{\psi},
\end{multline}
\begin{multline}\label{Ju28aa}
\mathcal{G}(u,\zeta):=\partial_tG(\hat{\psi}+tu,\eta+t\zeta)|_{t=0}=\lambda^2\Bigl(\partial_{\hat{x}}\hat{\psi}-\frac{\hat{y}\eta'}{\eta}\partial_{\hat{y}}\hat{\psi}\Bigr)\Bigl(\partial_{\hat{x}}u
-\frac{\hat{y}\eta'}{\eta}\partial_{\hat{y}}u\Bigr)\\
{}+\frac{d^2}{\eta^2}\partial_{\hat{y}}\hat{\psi}\partial_{\hat{y}}u+\zeta
-\lambda^2\Bigl(\partial_{\hat{x}}\hat{\psi}-\frac{\hat{y}\eta'}{\eta}\partial_{\hat{y}}\hat{\psi}\Bigr)\Bigl(\frac{\zeta}{\eta}\Bigr)'\hat{y}\partial_{\hat{y}}\hat{\psi}
-\frac{d^2\zeta}{\eta^3}\partial_{\hat{y}}\hat{\psi}\partial_{\hat{y}}\hat{\psi}
\end{multline}
and
\begin{equation}\label{Ju28ab}
\mathcal{H}u=u|_{\hat{y}=d}.
\end{equation}
Here $u=0$ for $\hat{y}=0$.

Let us introduce the transformation
\begin{equation}\label{Au2a}
v(x,y)=u(\hat{x},\hat{y})-\psi_{y}(x,y)\frac{y\zeta}{\eta}.
\end{equation}

\begin{lemma} $(i)$ Assume that the functions $\psi$ and $\eta$ satisfy the first equation\/ \eqref{Okt31a} in the domain $D_\eta$. If the function  $v$ is given by\/ \eqref{Au2a}, then
$$
(\lambda^2\partial_x^2+\partial_y^2)v+\omega'(\psi)v=\mathcal{F}(u,\zeta).
$$

$(ii)$ Furthermore
$$
\lambda^2\psi_xv_x+\psi_yv_y+\widehat{\rho}\zeta=\mathcal{G}(u,\zeta)\quad\mbox{\!on\!}\quad \mathcal{S}_\eta.
$$

\end{lemma}
\begin{proof} $(i)$ Using relations (\ref{Ju28a}) and (\ref{Ju28ab}), we get
\begin{align*}
(\lambda^2\partial_x^2+\partial_y^2)v+\omega'(\psi)v={}&\Bigl(\lambda^2\Bigl(\partial_{\hat{x}}-\frac{\hat{y}\eta'}{\eta}\partial_{\hat{y}}\Bigr)^{\!2}
+\Bigl(\frac{d}{\eta}\partial_{\hat{y}}\Bigr)^{\!2} +\omega'(\hat{\psi})\Bigr)u(\hat{x},\hat{y})\\
&{}-\Bigl(\lambda^2\partial_x^2+\partial_y^2+\omega'(\psi)\Bigr)\Bigl(\psi_{y}(x,y)\frac{y\zeta}{\eta}\Bigr)\\
{}={}&\Bigl(\lambda^2\Bigl(\partial_{\hat{x}}-\frac{\hat{y}\eta'}{\eta}\partial_{\hat{y}}\Bigr)^{\!2}+\Bigl(\frac{d}{\eta}\partial_{\hat{y}}\Bigr)^{\!2} +\omega'(\hat{\psi})\Bigr)u(\hat{x},\hat{y})\\
&{}-\frac{y\zeta}{\eta}\Bigl(\lambda^2\partial_x^2+\partial_y^2+\omega'(\psi)\Bigr)\psi_{y}(x,y)-I,
\end{align*}
where
\begin{equation*}
I=\lambda^2\Bigl(2(y\partial_y\psi_x\Bigl(\frac{\zeta}{\eta}\Bigr)'+y\psi_y\Bigl(\frac{\zeta}{\eta}\Bigr)''\Bigr)+2\psi_{yy}\frac{\zeta}{\eta}.
\end{equation*}
 Comparing this with the second line in (\ref{Ju28a}), we arrive at the assertion (i).

$(ii)$ We have
\begin{align*}
\lambda^2\psi_xv_x+\psi_yv_y+\widehat{\rho}\zeta={}&\lambda^2\psi_xu_x+\psi_yu_y+\widehat{\rho}\zeta\\
&{}-(\lambda^2\psi_x\psi_{xy}+\psi_y\psi_{yy})\zeta-\lambda^2\psi_x\psi_y y\Bigl(\frac{\zeta}{\eta}\Bigr)'-\psi_y^2\frac{\zeta}{\eta}\\
{}={}&\lambda^2\psi_xu_x+\psi_yu_y+\zeta-\lambda^2\psi_x\psi_y y\Bigl(\frac{\zeta}{\eta}\Bigr)'-\psi_y^2\frac{\zeta}{\eta}.
\end{align*}
Together with (\ref{Ju28aa}), this leads to the required proof of $(ii)$.
\end{proof}

\begin{corollary} Let the functions $\psi$ and $\xi$ satisfy the first equation\/ \eqref{K2aa} in the domain~$Q$.
Assume that\/ $u$ and\/ $\zeta$ satisfy
\begin{equation*} 
\begin{aligned}
&\mathcal{F}(u,\zeta)=0\quad\mbox{\!in\!}\quad Q,\\
&\mathcal{G}(u,\zeta)=\mu b\zeta\quad\mbox{\!for\!}\quad y=d,\\
&u=0\quad\mbox{\!for\!}\quad y=d.
\end{aligned}
\end{equation*}
Then the functions\/ $v$ and\/ $\zeta$ satisfy
\begin{equation*} 
\begin{aligned}
& (\lambda^2\partial_x^2+\partial_y^2+\omega'(\psi))v=0\quad\mbox{\!in\!}\quad D_\eta,\\
&(\lambda^2\psi_x\partial_x+\psi_y\partial_y) v-\frac{\widehat{\rho}}{\psi_y} v=\mu b v\quad\mbox{\!on\!}\quad S_\eta,\\
&v=0\quad\mbox{\!for\!}\quad y=0,
\end{aligned}
\end{equation*}
and
$$
\zeta=-v/\psi_y \quad\mbox{\!on\!}\quad S_\eta.
$$
\end{corollary}

If we choose here $b=1/\psi_y$, then we arrive at the spectral problem (\ref{Au5a}), (\ref{Dec27a}).

\section{Constant vorticity}

In this section we assume that $\omega=a$, where $a$ is a constant.

\subsection{Uniform stream solution, dispersion equation}\label{SAu13a}

 The uniform stream solution $U$ and the depth $d$ satisfy
\begin{gather*}
U''(y)+a=0\quad\mbox{on\!}\quad (0,d),
\\
U(0)=0,\quad U(d)=1,
\end{gather*}
and
$$
\frac{1}{2}U'(d)^2+d=R.
$$
Solving these equations, we get
\begin{equation}\label{eq:defU}
U(y)=-\frac{a}{2}y(y-d)+\frac{y}{d},
\end{equation}
where $d$ satisfies 
\begin{equation}\label{Ja4a}
\mathcal{R}(d)=R,\quad \mathcal{R}(d):=\frac{1}{2}\Bigl(\frac{1}{d^2}-a+\frac{a^2d^2}{4}\Bigr)+d.
\end{equation}
The function $\mathcal{R}$ is convex and its minimum can be found from
\begin{equation}\label{Ju8b}
\mathcal{R}'(d)=-\frac{1}{d^3}+\frac{a^2d}{4}+1=-\frac{1}{d^3}\Bigl(1-\frac{a^2d^4}{4}\Bigr)+1=0.
\end{equation}
We denote by $\dcrit$ the value of $d$ where the minimum is attained. Writing the equation for $\dcrit$ in the form
\[
\dcrit^3=\frac{4}{a^2\dcrit+4},
\]
it is easy to see that $\dcrit^3<1$ if $a\neq 0$, and $\dcrit=1$ if $a=0$ (see Fig.~\ref{fig:mu2_pos_cf}). Besides, one can verify that $\dcrit(a)<\ds(a)=\sqrt{2/|a|}$. Indeed, the last displayed formula can be written as $a^2 \dcrit^4=4(1-\dcrit^3)$ which yields the inequality $|a| \dcrit^2=2\sqrt{1-\dcrit^3}<2$. Thus, we find $\dcrit^2<2/|a|=\ds^2$.

To introduce the corresponding dispersion equation, we consider the solution $\gamm=\gamm(y;\tau)$ of the following Dirichlet problem
\begin{equation*}
\gamm''-\tau^2\gamm=0\quad\mbox{on\!}\quad (0,d),
\end{equation*}
and
\begin{equation*}
\gamm(0;\tau)=0,\quad\gamm(d;\tau)=1.
\end{equation*}
Solving this problem, we get
$$
\gamm(y;\tau)=\frac{\sinh(\tau y)}{\sinh(\tau d)}.
$$
To define the dispersion equation, we depart from solving the boundary value problem
\begin{equation}\label{Fe5a}
\begin{aligned}
&(\partial_x^2+\partial_y^2)u=0\quad\mbox{in\!}\quad\mathbb{R}\times (0,d),\\
&\dUd\partial_yu-\frac{\widehat{\rho}_0}{\dUd}u=0\quad\mbox{for\!}\quad y=d,\\
&u(x,0)=0.
\end{aligned}
\end{equation}
Here the left-hand side is the Fr\'{e}chet derivative at the laminar solution $(U,d)$, $\lambda=1$ and
$$
\dUd=U'(d)=\frac{1}{d}-\frac{ad}{2},\quad\widehat{\rho}_0=1-\mddU\dUd.
$$
Seeking even, periodic solutions in the form
$$
u(x,y)=\gamm(y;\tau)\cos \tau x,
$$
we arrive at  the dispersion equation
\begin{equation}\label{Ju8ba}
\DispOld(\tau):=\dUd\gamm'(y;\tau)-\frac{\widehat{\rho}_0}{\dUd}=0,
\end{equation}
which guarantees that $u$ solves (\ref{Fe5a}). Further it will be convenient to use a slightly modified variant of the equation \eqref{Ju8ba}
\begin{equation}\label{eq:dispdef}
\Disp(\tau):=\dUd^2\gamm'(y;\tau)+\mddU\dUd-1=0.
\end{equation}

We note that $\dUd=0$ for $a>0$ and $d=\ds(a)$. In the other case, when $\dUd\neq0$, we have
$$
\Disp'(\tau)=\frac{\dUd^2(\cosh(d\tau)\sinh(d\tau)-d\tau)}{\sinh(d\tau)^2}>0,
$$
i.e.\ the function $\Disp(\tau)$ is strictly increasing. At the same time, for $d>\dcrit$
$$
\Disp(0)=-\mathcal{R}'(d)<-\mathcal{R}'(\dcrit)=0.
$$
Thus, for $\dUd\neq0$ equation (\ref{eq:dispdef}) has a unique positive solution, which we denote by
$\taust$ and put
\begin{equation*} 
\Lambda_0=\Lambda_*=\frac{2\pi}{\taust}.
\end{equation*}

The case, when $U'(y)$ changes  sign, is of particular interest for us. Assume that 
\begin{equation}
U'(y_*)=0\quad\mbox{for some $y_*\in(0,d)$}.
\label{eq:U'=0}
\end{equation}
Since $U'' = -a$, the case $a < 0$ yields strict convexity of $U(y)$, making $y_*$ the unique minimum on $[0,d]$ and implying $U(y_*) < U(0) = 0$. Conversely, for $a > 0$, strict concavity ensures that, provided $y_* \in (0,d)$, the profile attains the unique maximum at $y_*$, which yields $U(y_*) > U(d) = 1$.

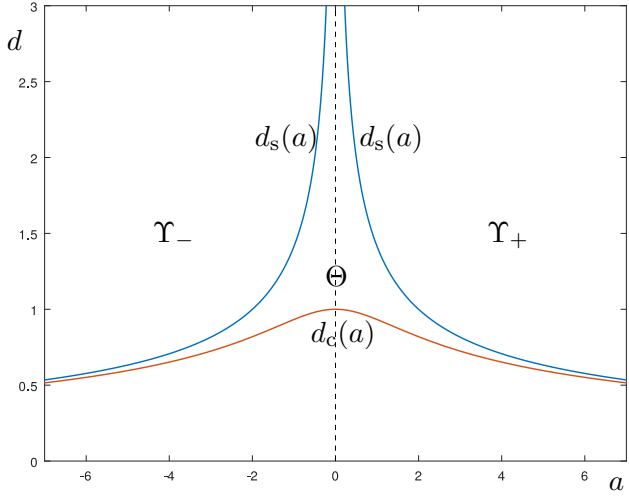
\begin{figure}[t!]
\centering
\centering\vspace{1.25mm}
 \SetLabels
 \L (-0.02*0.9) $d$\\
 \L (0.97*-0.02) $a$\\
 \L (0.48*0.29) $\dcrit(a)$\\
 \L (0.38*0.7) $\ds(a)$\\
 \L (0.56*0.7) $\ds(a)$\\
 \L (0.22*0.5) $\Upsilon_-$\\
 \L (0.77*0.5) $\Upsilon_+$\\
 \L (0.5037*0.406) $\Theta$\\
 \endSetLabels
 \leavevmode\AffixLabels{\includegraphics[width=80mm]{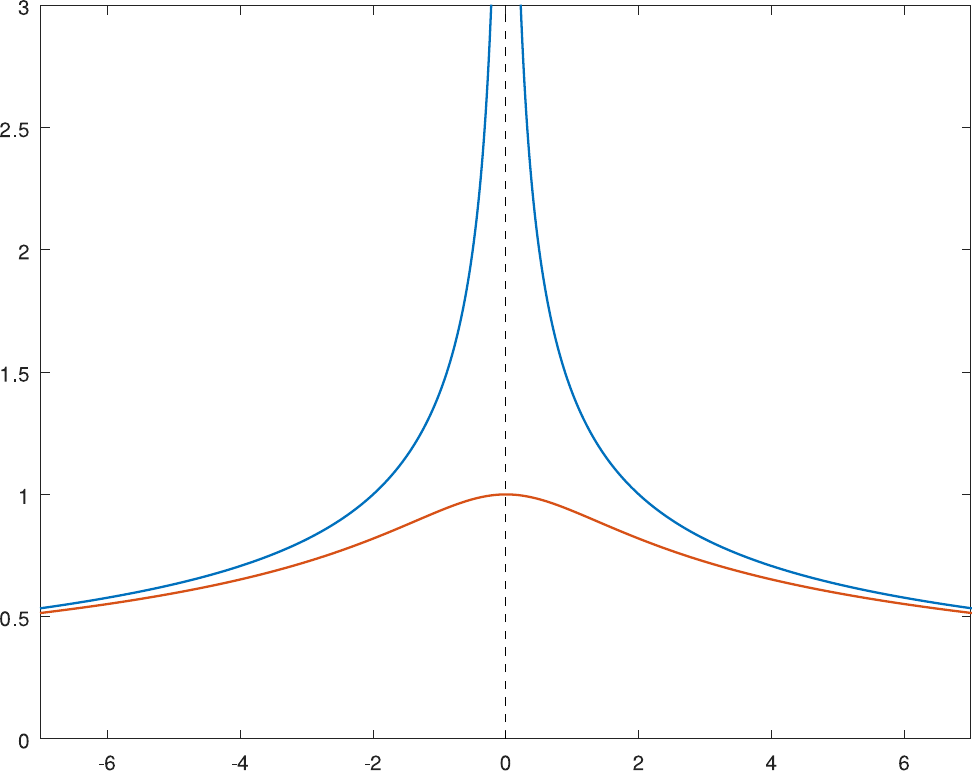}}\vspace{1mm}
 \caption{Stagnation points of the laminar flow $U$: stagnation points are absent in the domain $\Theta$ and they are present in both $\Upsilon_-$ and $\Upsilon_+$.}
\label{fig:Udiff0}
\end{figure}

From \eqref{eq:U'=0} 
\begin{equation*}
y_*=\frac{a d^2+2}{2ad}\quad\mbox{or}\quad
Y_*=Y_*(\varsigma):=\frac{y_*}{d}=\frac{\varsigma+2}{2\varsigma},\quad\mbox{where}\quad \varsigma=ad^2.
\end{equation*}
In the case $a>0$, $Y_*\in(1/2,1]$ for $\varsigma\in[2,+\infty)$. The function $Y_*(\varsigma)$ decreases monotonically from the value $1$ at $\varsigma=2$ to $1/2$, which is the limit of $Y_*$ as $\varsigma\to+\infty$.
In the case $a<0$, $Y_*\in[0,1/2)$ for $\varsigma\in(-\infty,-2]$, where $Y_*(\varsigma)\to 1/2$ as $\varsigma\to-\infty$ and $Y_*(\varsigma)$ decreases monotonically to zero at $\varsigma=-2$.

For $a>0$, the equality $\varsigma=2$ leads to $d=\ds(a)$  and, therefore, the equality $\Disp(\tau)=-1$ ($\dUd=0$) implies that the extremum point \eqref{eq:U'=0} is located on the free surface. If $\varsigma=-2$, which implies that $a<0$, the extremum point is located on the bottom and $d=\ds(a)$.

These observations are illustrated in Fig.~\ref{fig:Udiff0}. There, $\Theta=\{(a,d): \dcrit(a)<d<\ds(a)\}$ is the domain, where $y_*$  does not belong to $[0,d]$. In $\Upsilon_+=\{(a,d): {a>0},\,d>\ds(a)\}$, there is  $y_*\in(d/2,d)$, and
in $\Upsilon_-=\{(a,d):a<0,\,d>\ds(a)\}$, there is $y_*\in(0,d/2)$.

\begin{remark}
Let us give some details on computation of the function $\dcrit(a)$ shown in Figs.~\ref{fig:mu2_pos_cf} and \ref{fig:Udiff0}, which is defined as the real positive root of the equation \eqref{Ju8b}. Dividing \eqref{Ju8b} by $\dcrit$, we arrive at the equation
\begin{equation}
s^4-s-p=0,\quad\mbox{where}\ \ s=\frac{1}{\dcrit},\ \ p=\frac14a^2.\label{eq:seq}
\end{equation}
A solution of \eqref{eq:seq} is
\begin{equation}
s=\frac12\Delta+\frac12\sqrt{\frac{2}{\Delta}-\Delta^2},\quad \mbox{where}\ \ \Delta=\sqrt{2q},
\label{eq:sqdef}
\end{equation}
and $q$ is the real solution of the resolvent cubic equation $E(p,q):=8q^3+8p q-1=0$, so
\begin{equation}
q=r-\frac{p}{3r},\quad r=\left(\frac{\sqrt{27+256p^3}}{16\cdot3^{3/2}}+\frac{1}{16}\right)^{\!1/3}.
\label{eq:qr}
\end{equation}
It is not difficult to see that $r^2\geq p/3$ (dropping the terms $27$ and $1/16$), which, in turn, guarantees that $q\geq 0$. Besides, $q=1/2$ for $p=0$, and $q'(p)=-\pd{p}E(p,q)/\pd{q}E(p,q)=-q/(3q^2+p)<0$, which means that $q\in[0,1/2]$. Then, the expression under square root in \eqref{eq:sqdef} is positive and $s$ is real and positive. Finally, combining \eqref{eq:qr} with \eqref{eq:seq} and \eqref{eq:sqdef}, we arrive at a rather simple representation of $\dcrit(a)$. It also help us obtain the following asymptotics, which will be useful in the subsequent analysis:
\begin{equation}
\dcrit(a) = \frac{2^{1/2}}{|a|^{1/2}} - \frac{1}{a^2} + \frac{3}{2^{3/2}|a|^{7/2}}-\frac{1}{|a|^5}
+\Obig\bigl(a^{-13/2}\bigr)\quad\mbox{as\!}\quad a\to\infty.
\label{eq:dcritasympt}
\end{equation}
\end{remark}


\subsection{Branch of Stokes waves. The first eigenvalue of the Fr\'{e}chet derivative}\label{SecJan16b}

Let a subcritical laminar flow $(U(y),d)$  be fixed, $d>\dcrit$, $\dUd\neq 0$ and  $\taust$ be the root of the dispersion equation (\ref{eq:dispdef}). Then there exists a branch of even, $\Lambda_*$-periodic Stokes waves
\begin{equation}\label{Dec28ba}
(\psi(x,y;t),\eta(x;t),\lambda(t))\;\in C^{2,\alpha}(D_\eta)\times C^{2,\alpha}(\mathbb{R})\times (0,\infty),\quad t\geq 0.
\end{equation}
 The following properties (see \cite{Var23} and \cite{Koz2}) are valid for the branch:

(i) The vector function
$$
[0,\infty) \ni t\rightarrow \Bigl(\psi\Bigl(\hat{x},\frac{\eta(\hat{x};t)\hat{y}}{d};t\Bigr),\eta(\hat{x};t),\lambda(t)\Bigr)\in C^{2,\alpha}(Q)\times C^{2,\alpha}(\mathbb{R})\times (0,\infty)
$$
is analytic with respect to $t$ (possibly after local reparametrization). Here the variables $\hat{x}$ and $\hat{y}$ were introduced in Sect.~\ref{SecJan16a}.

(ii) The vector function (\ref{Dec28ba}) is analytic in a neighborhood of every point $(x,y,t)\in \overline{D_\eta}\times [0,\infty)$
and
\begin{align*}
&\psi(x,y;t)=U(y)-\dUd \gamm(y;\taust)\cos(\taust x)+\Obig(t^2),\\
&\eta(x;t)=d+t\cos(\taust x)+\Obig(t^2),\\
&\lambda(t)=1+\lambda_2t^2+\Obig(t^3)
\end{align*}
for small $t$.

Denote by $\mu_1(t)$ the first eigenvalue of the spectral problem (\ref{Au5a}), (\ref{Dec27a}), where $(\psi,\eta,\lambda)=(\psi(t),\eta(t),\lambda(t))$, $t\geq 0$.

\begin{proposition}
The first eigenvalue $\mu_1(t)$ of the Fr\'{e}chet derivative is negative for all\/ $t\geq 0$.
\end{proposition}

\begin{proof}
The function $v(x,y;t)=\psi_x(x,y;t)$ is odd with respect to $x$ and satisfies the spectral problem \eqref{Au5a}, \eqref{Dec27a} for every $t\geq 0$, where $\mu=0$.

Let
$$
D_0=\{(x,y)\in D_\eta\,:\, 0<x<\Lambda_*/2\},\quad S_0=\{(x,y)\in S_\eta\,:\, 0<x<\Lambda_*/2\}.
$$
Then the part of the boundary $\partial D_0\cap \{x=0\}$ is an interval $[0,y_0]$ and the part of the boundary
$\partial D_0\cap \{x=\Lambda_*/2\}$ is an interval $[0,y_1]$.
In \cite{Koz2} it is proved that
$v(x,y;t)>0$ for $x\in (0,\Lambda_*/2)$ and $y\in (0,\eta(x;t))$,
\begin{equation*}
v(0,y;t)=0\quad\mbox{on}\quad (0,y_0)\quad \mbox{and}\quad v(\Lambda_*/2,y;t)=0\quad\mbox{on}\quad (0,y_1).
\end{equation*}
Thus $v$ possesses nodal lines at the vertical boundaries of $D_0$.

The first eigenfunction $u$ satisfies also the eigenvalue problem \eqref{Au5a}, \eqref{Dec27a} with an eigenvalue $\mu_1(t)$. 
By the standard theory of self-adjoint elliptic operators, the principal eigenfunction $u$ can be chosen sign-definite 
(i.e., it has no nodes in the fundamental domain $D_0$). Due to the symmetry of the problem, $u$ is even and hence satisfies
\begin{equation*}
u_x(0,y)=0\quad\mbox{on}\quad (0,y_0)\quad \mbox{and}\quad u_x(\Lambda_*/2,y)=0\quad\mbox{on}\quad (0,y_1).
\end{equation*}

Now we distinguish two cases:

Case $t > 0$: The function $v = \psi_x$ is non-trivial and has nodal lines, while the principal eigenfunction $u$ has no nodes. 
By the variational principle (or Courant nodal domain theorem), the principal eigenvalue $\mu_1(t)$ is strictly less than any 
eigenvalue associated with a nodal eigenfunction. Since $v$ corresponds to $\mu = 0$, we conclude $\mu_1(t) < 0$.

Case $t = 0$: At the laminar flow, $\psi_x \equiv 0$, so the above argument does not apply directly. 
However, direct computation gives $\mu_1(0) = \sigma(0) < 0$ (see Sect.~\ref{SectJ22a}).

Thus, $\mu_1(t) < 0$ for all $t \ge 0$.
\end{proof}

\subsection{Exchange of stability}\label{SectJ22a}

Consider the problem \eqref{Okt31a} for even, $\Lambda_*$-periodic functions.
The function $\psi=U$ satisfies this problem for all $\lambda>0$ if $R$ is defined by (\ref{Ja4a}). Let us calculate the Fr\'{e}chet derivative.  According to (\ref{Au5a}), (\ref{Dec27a}) the corresponding spectral problem is defined by
\begin{equation}
\label{Ja22a}
\begin{aligned}
&\lambda^2\partial_x^2u+\partial_y^2u=0\quad\mbox{for\!}\quad x\in\mathbb{R},\ \ y\in (0,d),\\
&\dUd^2 u'-(1-\mddU\dUd)u=\mu u\quad\mbox{for\!}\quad y=d,\\
&u(x,0)=0.
\end{aligned}
\end{equation}
Here $u$ is an even $\Lambda_*$-periodic function. This problem can be solved by using separation of variables method:
$$
u(x,y)=\sum_{k=0}^\infty h_k(y)\cos(k\taust x).
$$
Then we obtain the following boundary value problems for the coefficients $h_k$
\begin{align*}
&h_k''(y)-\lambda^2k^2\taust^2h_k=0\quad\mbox{on\!}\quad (0,d),\\
&\dUd^2 h_k'-(1-\mddU\dUd)h_k=\mu_k h_k\quad\mbox{for\!}\quad y=d,\\
& h_k(0)=0.
\end{align*}
Solution to this problem is given by
$$
\gamm_k(y)=c_k\gamm(y;\lambda k\taust))\quad\mbox{and}\quad\mu_{1+k}(\lambda)=\Disp(\lambda k\taust),
$$
which imply
$$
\mu_1=\Disp(0),\quad\mu_2(\lambda)=\Disp(\lambda^2\taust)\quad\mbox{and}\quad \mu_k(\lambda)\geq \Disp(4\lambda^2\taust^2)\;\;\mbox{for $k>2$}.
$$
Therefore the first eigenvalue is always negative and the second eigenvalue $\mu_1(\lambda)$ is zero for $\lambda=1$, it is positive for $\lambda>1$ and it is negative for $\lambda<1$.

If we consider these eigenvalues as functions of the period $\Lambda=2\pi/\tau$ then $\mu_1(\Lambda)$ is always negative, $\mu_2(\Lambda)$ is zero for $\Lambda=\Lambda_*$, it is positive for $\Lambda<\Lambda_*$ and it is negative for $\Lambda>\Lambda_*$. So if $\Lambda$ increases from $0$ to $\infty$ then at the point $\Lambda_*$ appears additional negative eigenvalue, which gives additional instability to the problem. This leads to appearance of a new bifurcation branch  at the point $\Lambda_*$ and our aim is to verify the principle of exchange of stability with respect to small eigenvalues near this point.

\subsection{Asymptotic analysis}\label{SJa22b}

Let $(U,d)$ be a subcritical laminar flow with $R=\mathcal{R}(d)$. This means that $d>\dcrit$ and $U$ is given by (\ref{eq:defU}). Since $d>\dcrit$, the dispersion equation (\ref{eq:dispdef}) is uniquely solvable, and $\taust$ is  its positive root. We also keep the notations $\Lambda_0=\Lambda_*$ and $\gamm(y;\tau)$ introduced in Sect.~\ref{SAu13a}.

We write the problem \eqref{Okt31a} as follows:
\begin{align}
&\bigl(\lambda^2 \pd[2]{x}+\pd[2]{y}\bigr)\psi(x,y)+a=0\quad\mbox{in\!}\quad D_\eta,\label{eq:1}\\
&\psi(x,0)=0\quad\mbox{for\!}\quad x\in\mathbb{R},\label{eq:2}\\
&\psi(x,\eta(x))-1=0\quad\mbox{for\!}\quad x\in\mathbb{R},\label{eq:3-0}\\
&\frac12\left((\pd{y}\psi)^2+\lambda^2(\pd{x}\psi)^2\right)+\eta(x)-R=0\quad\mbox{on\!}\quad S_\eta.\label{eq:4-0}
\end{align}
We are looking for the solution $\eta$, $\psi$ and $\lambda$ of \eqref{eq:1}--\eqref{eq:4-0} in the form
\begin{equation}\label{eq:etaexpdef}
\eta(x):=d+\zeta(x),\quad\mbox{where}\ \
\zeta(x):=t \eta_0(x) + t^2 \eta_1(x) + t^3 \eta_2(x)+\Obig(t^4),
\end{equation}
\begin{equation}
\psi(x,y):=U(y)+t \psi_0(x,y)+t^2 \psi_1(x,y)+t^3 \psi_2(x,y)+\Obig(t^4)
\label{eq:psiexpdef0}
\end{equation}
and
\begin{equation*}
\lambda=\lambda(t)=1+\lambda_2t^2+\Obig(t^4).
\end{equation*}

As it is assumed that $\zeta(x)$ is small, we write \eqref{eq:3-0} as Taylor series:
\begin{equation}
1 = \psi(x,d) + \zeta(x)\pd{y}\psi(x,y)\bigr|_{y=d} + \tfrac12\zeta(x)^2 \pd[2]{y}\psi(x,y)\bigr|_{y=d}
+ \tfrac16\zeta(x)^3 \pd[3]{y}\psi(x,y)+\dots
\label{eq:3}
\end{equation}

Analogously,  \eqref{eq:4-0} can be written as
\begin{align}
\pd{x}\psi(x,y)\bigr|_{y=\eta(x)}={}&
\pd{x}\psi(x,d) + \zeta(x)\pd{y}\pd{x}\psi(x,y)\bigr|_{y=d} \notag\\ &{}+
\tfrac12\zeta(x)^2\pd[2]{y}\pd{x}\psi(x,y)\bigr|_{y=d} +
\tfrac16\zeta(x)^3\pd[3]{y}\pd{x}\psi(x,y)\bigr|_{y=d}+\dots,
\label{eq:psix}\\[2mm]
\pd{y}\psi(x,y)\bigr|_{y=\eta(x)}={}&
\pd{y}\psi(x,y)\bigr|_{y=d} + \zeta(x)\pd[2]{y}\psi(x,y)\bigr|_{y=d} \notag\\ &{}+
\tfrac12\zeta(x)^2\pd[3]{y}\psi(x,y)\bigr|_{y=d} +
\tfrac16\zeta(x)^3\pd[4]{y}\psi(x,y)\bigr|_{y=d}+\dots.
\label{eq:psiy}
\end{align}
Now, collect terms of the same order ($\Obig(t^0)$, $\Obig(t^1)$, $\Obig(t^2)$, and $\Obig(t^3)$) in equations \eqref{eq:1}, \eqref{eq:2}, \eqref{eq:3}, and \eqref{eq:4-0} after the substitution of \eqref{eq:psix} and \eqref{eq:psiy}. It is straightforward to verify that the term of order $\Obig(t^0)$ appears only in \eqref{eq:4-0} and coincides with \eqref{Ja4a}.

By taking the terms of order $\Obig(t^1)$ in \eqref{eq:1}, \eqref{eq:2}, \eqref{eq:3}, and \eqref{eq:4-0}, we obtain
\begin{align}
&\bigl(\pd[2]{x}+\pd[2]{y}\bigr)\psi_{0}(x,y)=0\quad\mbox{for\!}\quad y\in (0,d),\label{eq:10}\\
&\psi_0(x,0)=0,\label{eq:20}\\
&\psi_0(x,d)+\dUd\eta_0(x)=0\quad \mbox{for\!}\quad y=d,\label{eq:30}\\
&\dUd\pd{y}\psi_0(x,y)+(1-\mddU\dUd)\eta_0(x)=0\quad \mbox{for\!}\quad y=d.\label{eq:40}
\end{align}
The solution is sought in the form
\begin{equation}
\eta_0(x):=\cos(\taust x),\quad
\psi_0(x,y):=-\dUd\cos(\taust x)\gamm(y;\taust)
\label{eq:eta_psi0}
\end{equation}
to satisfy \eqref{eq:10}--\eqref{eq:30}. From \eqref{eq:40}, we obtain
\begin{equation*}
(1-\mddU\dUd)\cos(\taust x)-\gamm'(d;\taust)\dUd^2\cos(\taust x)=0.
\end{equation*}
This equality holds because of the dispersion relation $\Disp(\taust)=0$ (see \eqref{eq:dispdef}).

Let us now take the terms of order $\Obig(t^2)$ in \eqref{eq:1}, \eqref{eq:2}, \eqref{eq:3}, and \eqref{eq:4-0}. This leads to the problem
\begin{align}
&\bigl(\pd[2]{x}+\pd[2]{y}\bigr)\psi_{1}(x , y)=0\quad\mbox{for\!}\quad y\in (0,d),\label{eq:11}\\
&\psi_1(x,0)=0,\label{eq:21}\\
&\psi_1(x,d)+\dUd\eta_1(x)+\eta_0(x)\pd{y}\psi_0(x,y)-\tfrac12 \mddU[\eta_0(x)]^2=0\quad \mbox{for\!}\quad y=d,\label{eq:31_}\\
&\dUd\pd{y}\psi_1(x,y)+(1-\mddU\dUd)\eta_1(x) +\eta_0(x)
\bigl(\dUd\pd[2]{y}-\mddU\pd{y}\bigr)\psi_0(x,y) \notag\\
&\kern4mm{}+\tfrac12[\pd{y}\psi_0(x,y)]^2 +
\tfrac12[\pd{x}\psi_0(x,d)]^2+\tfrac12[\mddU\eta_0(x)]^2
=0\quad \mbox{for\!}\quad y=d.\label{eq:41_}
\end{align}
After substitution of \eqref{eq:eta_psi0} into  equations \eqref{eq:31_} and \eqref{eq:41_}, we get
\begin{gather}
\psi_1(x,d)+\dUd\eta_1(x)+A_1(1+\cos(2\taust x))=0,\label{eq:31}\\
\dUd\pd{y}\psi_1(x,y)\bigr|_{y=d}+(1-\mddU\dUd)\eta_1(x)+B_1(1+\cos(2\taust x))+C_1=0,\label{eq:41}
\end{gather}
where
\begin{equation*}
\begin{gathered}
A_1:=-\tfrac{1}{4}\mddU - \tfrac{1}{2}\dUd\gamm'(d;\taust),\quad C_1:=\tfrac{1}{2}\taust^2 \dUd^2,\\
B_1:=-\tfrac{3}{4}\taust^2\dUd^2+\tfrac{1}{4}\mddU^2+\tfrac{1}{2}\mddU\dUd\gamm'(d;\taust)+\tfrac{1}{4}\dUd^2\gamm'(d;\taust)^2.
\end{gathered}
\end{equation*}

We seek solution in the form
\begin{equation}
\label{eq:eta_psi1}
\begin{gathered}
\eta_1(x):=a_1 + b_1 \cos(2\taust x);\\
\psi_1(x,y):=c_1 y + d_1 \cos(2\taust x) \gamm(y;2\taust),
\end{gathered}
\end{equation}
which obviously satisfies \eqref{eq:11}, \eqref{eq:21}. Substituting \eqref{eq:eta_psi1} into \eqref{eq:31} and  \eqref{eq:41}, and collecting coefficients of $\cos(2\taust x)$ and of $1$, we obtain two linear systems for the unknown coefficients $a_1$, $c_1$, and $b_1$, $d_1$:
\begin{gather*}
\dUd a_1+dc_1+A_1=0,\\
(1-\mddU\dUd)a_1+\dUd c_1+B_1+C_1=0,
\end{gather*}
and
\begin{gather*}
\dUd b_1+d_1+A_1=0,\\
(1-\mddU\dUd)b_1+\dUd\gamm'(d;2\taust)d_1+B_1=0.
\end{gather*}
Therefore
\begin{equation}
\begin{aligned}
a_1&=[d\Disp(0)]^{-1}\bigl(A_1\dUd-(B_1+C_1)d\bigr),\\
b_1&=[\Disp(2\taust)]^{-1}\bigl(A_1\dUd\gamm'(d;2\taust)-B_1\bigr),\\
c_1&=[d\Disp(0)]^{-1}\bigl(\dUd(B_1+C_1)-A_1(1-\mddU\dUd)\bigr),\\
d_1&=[\Disp(2\taust)]^{-1}\bigl(B_1-A_1(1-\mddU\dUd)\bigr).
\end{aligned}\label{eq:abcd1}
\end{equation}

Finally, we consider the terms of order $\Obig(t^3)$ in \eqref{eq:1}, \eqref{eq:2}, \eqref{eq:3}, and \eqref{eq:4-0}, and obtain the following problem:
\begin{align}
&\bigl(\pd[2]{x}+\pd[2]{y}\bigr)\psi_{2}(x , y)+2\lambda_2\pd[2]{x}\psi_0(x,y)=0\quad\mbox{for\!}\quad y\in (0,d),\label{eq:12}\\[2mm]
&\psi_2(x,0)=0,\label{eq:22}\\[2mm]
&\psi_2(x,d)+\dUd\eta_2(x)+\eta_0(x)\pd{y}\psi_1(x,y)+\eta_1(x)\pd{y}\psi_0(x,y) \notag\\
&\kern4mm{} +\tfrac12[\eta_0(x)]^2\pd[2]{y}\psi_0(x,y) -
\mddU\eta_0(x)\eta_1(x)=0\quad \mbox{for\!}\quad y=d,\label{eq:32_}\\[2mm]
&\dUd\pd{y}\psi_2(x,y)+(1-\mddU\dUd)\eta_2(x)
+\eta_0(x)\bigl(\dUd\pd[2]{y}-\mddU\pd{y}\bigr)\psi_1(x,y)\notag\\
&\kern4mm{}
+\bigl(\pd{x}\psi_0(x,d)\pd{x}+\pd{y}\psi_0(x,y)\pd{y}\bigr)\psi_1(x,y) +
\mddU^2\eta_0(x)\eta_1(x)\notag\\
&\kern4mm{}
+\tfrac12 \dUd[\eta_0(x)]^2\pd[3]{y}\psi_0(x,y)+\bigl(-\mddU[\eta_0(x)]^2+\dUd\eta_1(x) \bigr)\pd[2]{y}\psi_0(x,y) \notag\\
&\kern4mm{}+
\eta_0(x)\bigl(
\pd{x}\psi_0(x,d)\pd{x}+\pd{y}\psi_0(x,y)\pd{y}\bigr)\pd{y}\psi_0(x,y) \notag\\
&\kern34mm{}-\mddU\eta_1(x)\pd{y}\psi_0(x,y)
=0\quad \mbox{for\!}\quad y=d.\label{eq:42_}
\end{align}
After substitution of solutions \eqref{eq:eta_psi0} and \eqref{eq:eta_psi1} into \eqref{eq:32_}, \eqref{eq:42_}, the equations
take the form
\begin{gather}
\psi_2(x,d)+\dUd\eta_2(x)+A_{2}\cos(\taust x)+B_{2}\cos(3\taust x)=0,\label{eq:32}\\
\dUd\pd{y}\psi_2(x,y)\bigr|_{y=d}+(1-\mddU\dUd)\eta_2(x)+C_2\cos(\taust x)+D_2\cos(3\taust x)=0,\label{eq:42}
\end{gather}
where 
\begin{equation}
\label{eq:ABCD2def}
\begin{aligned}
A_2:={}& (a_1+\tfrac12 b_1)\Xi +c_1
  + \tfrac12 \gamm'(d;2\taust)d_1 -\tfrac38\dUd \taust^2,\\
B_2:={}& \tfrac12 b_1\Xi + \tfrac12\gamm'(d;2\taust)d_1-\tfrac18\dUd \taust^2,\\
C_2:={}&  {-(a_1+\tfrac12 b_1)} ( \mddU\Xi + \dUd^2\taust^2 ) + c_1\Xi+ d_1 [\dUd\taust^2    + \tfrac12\Xi\gamm'(d;2\taust)]\\
&{} +\tfrac34\mddU\dUd\taust^2 + \tfrac58\dUd^2 \taust^2\gamm'(d;\taust) ,\\
D_2:={}&
{-\tfrac12 b_1 (\mddU\Xi + \dUd^2\taust^2)}
+ d_1\bigl[3\dUd\taust^2+ \tfrac12\Xi\gamm'(d;2\taust)\bigr]
+\tfrac14\mddU\dUd \taust^2  - \tfrac18\taust^2\dUd^2\gamm'(d;\taust)
\end{aligned}\!\!\!\!\!
\end{equation}
and  $\Xi:={-\mddU} - \dUd\gamm'(d;\taust)$.
Taking into account \eqref{eq:32} and \eqref{eq:42}, we seek $\eta_2$ in the form:
\begin{equation}
\eta_2(x):=a_2 \cos(\taust x) + b_2 \cos(3\taust x).\label{eq:eta2def}
\end{equation}
Here, in contrast to the above steps, the potential $\psi_2$ should satisfy the non-homogeneous equation \eqref{eq:12}. It is easy to verify that the function defined by\vskip-3mm
\begin{equation}
\psi_2(x,y):=
-\dUd\lambda_2\cos(\taust x) \,y\, \gamm'(y;\taust)
+c_2 \cos(\taust x)\gamm(y;\taust)
+d_2 \cos(3\taust x)\gamm(y;3\taust),\label{eq:psi2def}
\end{equation}
satisfies \eqref{eq:12} and \eqref{eq:22}.
Substituting \eqref{eq:eta2def} and \eqref{eq:psi2def} into \eqref{eq:32} and  \eqref{eq:42}, we collect the coefficients of
 $\cos(\taust x)$ and $\cos(3\taust x)$. Equating the coefficients to zero yields two linear systems for the unknown coefficients $a_2$, $\lambda_2$ (where $c_2$ is considered as a free parameter) and $b_2$, $d_2$:
\begin{gather*}
\dUd a_2- d\dUd\gamm'(d;\taust)\lambda_2+c_2+A_2=0,\\
(1-\mddU\dUd)a_2 - \dUd^2[d \taust^2 +\gamm'(d;\taust)]\lambda_2+\dUd\gamm'(d;\taust)c_2+C_2=0,
\end{gather*}
and\vskip-7mm
\begin{gather*}
\dUd b_2 + d_2 + B_2=0,\\
(1-\mddU\dUd)b_2+\dUd\gamm'(d;3\taust)d_2+D_2=0.
\end{gather*}
These systems have the following solutions:
\begin{align}
a_2&=-\frac{c_2}{\dUd}+\frac{-A_2 \dUd[d\taust^2+\gamm'(d;\taust)]+C_2d\gamm'(d;\taust)}
{\dUd^2[d\taust^2+\gamm'(d;\taust)]-d(1-\mddU\dUd)\gamm'(d;\taust)},\label{eq:a2def}\\
\lambda_2&=\frac{-(1-\mddU\dUd)A_2+C_2}
{\dUd^3[d\taust^2+\gamm'(d;\taust)]-d\dUd(1-\mddU\dUd)\gamm'(d;\taust)}.
\label{eq:lambda2def}\\
b_2&
 =[\Disp(3\taust)]^{-1}\bigl(-B_2\dUd\gamm'(d;3\taust)+D_2\bigr),\notag\\
d_2&=
[\Disp(3\taust)]^{-1}\bigl(B_2(1-\mddU\dUd)-\dUd D_2\bigr).\notag
\end{align}
The appearance of the free parameter $c_2$ in \eqref{eq:psi2def} and \eqref{eq:a2def} is rather natural,
it reflects the non-uniqueness in the choice of the parameter $t$ in the parametrization of the branch.

The lengthy algebraic manipulations of this section were carried out with the aid of Maxima computer algebra system; this applies to some derivations in subsequent sections, particularly those leading to \eqref{eq:mu2as_d_large}--\eqref{eq:mu2as_ds} and \eqref{eq:Basympt}.

\subsection{The second eigenvalue of the Fr\'{e}chet derivative (\ref{Au5a}), (\ref{Dec27a})}
\label{sect:mu2}

We write \eqref{eq:etaexpdef} and \eqref{eq:psiexpdef0} as
$$
\eta(x,t)=d+t\eta_*(x,t),\quad\eta_*(x,t)=\eta_0+t\eta_1(x)+t^2\eta_2(x)+\Obig(t^3)
$$
and
$$
\psi(x,y,t)=U(y)+t\psi_*(x,y,t),\quad\psi_*(x,y,t)=\psi_0(x,y)+t\psi_1(x,y)+t^2\psi_2(x,y)+\Obig(t^3).
$$
The eigenvalue problem (\ref{Au5a}), (\ref{Dec27a}) has the form
\begin{equation}
\label{eq:eigenproblem}
\begin{aligned}
&(\lambda^2\partial_x^2+\partial_y^2)v=0\quad\mbox{in\!}\quad D_\eta,\\
&\lambda^2\psi_xv_x+\psi_yv_y-(1+\lambda^2\psi_x\psi_{xy}+\psi_y\psi_{yy})\frac{v}{\psi_y}=\mu \frac{v}{\psi_y}\quad\mbox{for\!}\quad y=\eta(x),\\
&v(x,0)=0.
\end{aligned}
\end{equation}

Differentiating \eqref{eq:1}--\eqref{eq:4-0} with respect to $t$, we get
\begin{align*}
&(\lambda^2\partial_x^2+\partial_y^2)\psi_t+4\lambda_2t^2\partial_x^2\psi_*=\Obig(t^3),\\
&\psi_x(\psi_{xy}\eta_t+\psi_{xt})+\psi_y(\psi_{yy}\eta_t+\psi_{yt})+\eta_t
=\Obig(t^3)\quad\mbox{for\!}\quad y=\eta(x,t),\\
&\psi_y\eta_t+\psi_t=0\quad\mbox{for\!}\quad y=\eta(x,t),\\
&\psi(x,0,t)=0.
\end{align*}
Using that $\eta_t=-\psi_t/\psi_y$ and that $\psi_*=\psi_t+\Obig(t)$, we rewrite the problem with respect to the function $\psi_t$
\begin{equation}\label{Dec6c}
\begin{aligned}
&(\partial_x^2+\partial_y^2)\psi_t+6\lambda_2t^2\partial_x^2\psi_t=\Obig(t^3),\\
&\psi_x\partial_x\psi_{t}+\psi_y\partial_y\psi_{t}-\check{\rho}(x,t)\psi_t=\Obig(t^3)\quad\mbox{for\!}\quad y=\eta(x,t),\\
&\psi(x,0,t)=0,
\end{aligned}
\end{equation}
where
\begin{equation*}
\check{\rho}(x,t)=\frac{1+\psi_{xy}\psi_x+\psi_{yy}\psi_y}{\psi_y}.
\end{equation*}
To obtain the asymptotics of the problem \eqref{Dec6c}, we proceed as follows. Let
\begin{equation}
\label{eq:vdef}
v=\psi_t+h \cos(\taust x),\quad h=-2\lambda_2t^2\dUd \,y\,\frac{\taust\cosh(y\taust)}{\sinh(d\taust)}.
\end{equation}
Then, in view of \eqref{eq:eta_psi0}, \eqref{eq:12}, we have
\begin{equation*} 
((1+2\lambda_2t^2)\partial_x^2+\partial_y^2)v=\Obig(t^3),
\end{equation*}
and
$$
\psi_x\partial_xv+\psi_y\partial_yv-\check{\rho}(x,t)v-\Bigl(\dUd h_y-\frac{(1-a\dUd)}{\dUd}h\Bigr)\cos(\taust x)=\Obig(t^3).
$$
Since
$$
\dUd h_y-\frac{(1-a\dUd)}{\dUd}h=-2\dUd^2\taust\lambda_2t^2\Bigl(\taust d+\frac{\cosh\taust d}{\sinh\taust d}-\frac{(1-a\dUd)}{\dUd^2}d\frac{\cosh \taust d}{\sinh \taust d} \Bigr),
$$
we find
\begin{equation}
\mu_2=-2\dUd^2\taust\lambda_2\Bigl(\taust d+\Bigl(1-\frac{(1-a\dUd)d}{\dUd^2}\Bigr)\frac{\cosh \taust d}{\sinh \taust d} \Bigr).
\label{eq:mu2def}
\end{equation}
We have
\begin{equation}\label{Dec6ca}
H(z):=z+\Bigl(1-z\frac{\cosh z}{\sinh z}\Bigr)\frac{\cosh z}{\sinh z}>0\quad \mbox{for\!}\quad z>0,
\end{equation}
because $H(0)=0$ and
\begin{equation*}
H'(z)=\frac{2}{\sinh(z)^3}(z\cosh(z)-\sinh(z))>0\quad\mbox{for\!}\quad z>0.
\end{equation*}
Besides, due to the dispersion equation
$$
\frac{(1-a\dUd)d}{\dUd^2}=z\frac{\cosh z}{\sinh z},\quad z=d\taust.
$$
Therefore, we can write $\mu_2$ as
$$
\mu_2=-A\lambda_2,
$$
where 
$$
A=2\dUd^2\taust H(\taust d),
$$
and hence $A$ is positive due to positivity of the function $H$.

Consider now asymptotics of $\mu_2$. Using \eqref{eq:mu2def}, along with the expressions \eqref{eq:abcd1}, \eqref{eq:ABCD2def}, \eqref{eq:lambda2def} defining $\lambda_2$, one obtains dependence of $\mu_2$ on $a$, $d$ and $\taust$. Then, asymptotic analysis can be performed based on asymptotics of $\taust$ presented in Appendix~A.

First, let \mbox{$d\to\infty$} for a fixed value of \mbox{$a\neq0$}. Substituting the asymptotic representation \eqref{eq:taudefinf} of $\taust$  into the expression of $\mu_2$, we find
\begin{equation}
\mu_2 = m a^2 d^{-1}+\Obig(d^{-2})\quad\mbox{as\!}\quad d\to\infty,
\label{eq:mu2as_d_large}
\end{equation}
where
\begin{equation}
m=\frac{q_1^6-11\,q_1^4+28\, q_1^2-16}{8\,q_1^2}\approx -0.406748...,\label{eq:m1}
\end{equation}
and $q_1$ is defined by \eqref{eq:C1def}.

Now turn to the asymptotic expansion of $\mu_2$ near the critical depth $\dcrit$. We introduce the small parameter $\epsilon = d - \dcrit$, substitute the asymptotics for $\taust$ given in \eqref{eq:taucritexpansion} and expand the resulting expressions in a Taylor series in $\epsilon$. Throughout this process, we reduce even powers of the vorticity parameter by $a^{2} = 4(1-d_c^3)/d_c^4$, which is the critical depth equation \eqref{Ju8b}. As a result of these lengthy but straightforward symbolic computations, we arrive at
\begin{equation}
\mu_2=
\frac{5(4-\dcrit^3)}{12\,\dcrit^4}\,\epsilon^{-1}+
\frac{47\,\dcrit^6+15\,a\dcrit^5-361\,\dcrit^3-195\,a\dcrit^2+422}{30\dcrit^5(\dcrit^3+a\dcrit^2-2)}
+
\Obig(\epsilon),
\label{eq:mu2as_dcrit}
\end{equation}
Since $\dcrit\leq 1$, the coefficient of $\epsilon^{-1}$ is positive.

Let now $a>0$. Using \eqref{eq:taustasinf}, we can write the asymptotics of $\mu_2$ as $d\to \ds=\sqrt{2/a}$:
\begin{equation}
\mu_2=
-\frac{2}{a^4}(d-\ds)^{-4}+\Obig\bigl((d-\ds)^{-3}\bigr)\quad \mbox{as\!}\quad d-\ds\to0.
\label{eq:mu2as_ds}
\end{equation}


\begin{figure}[t!]
\centering\vspace{1.25mm}
 \SetLabels
 \L (-0.06*0.23) \rotatebox{90}{$\operatorname{sgn}(\mu_2(d))\log(1+|\mu_2(d)|)$}\\
 \L (0.9*-0.02) $d$\\
 \endSetLabels
 \leavevmode\AffixLabels{\includegraphics[width=90mm]{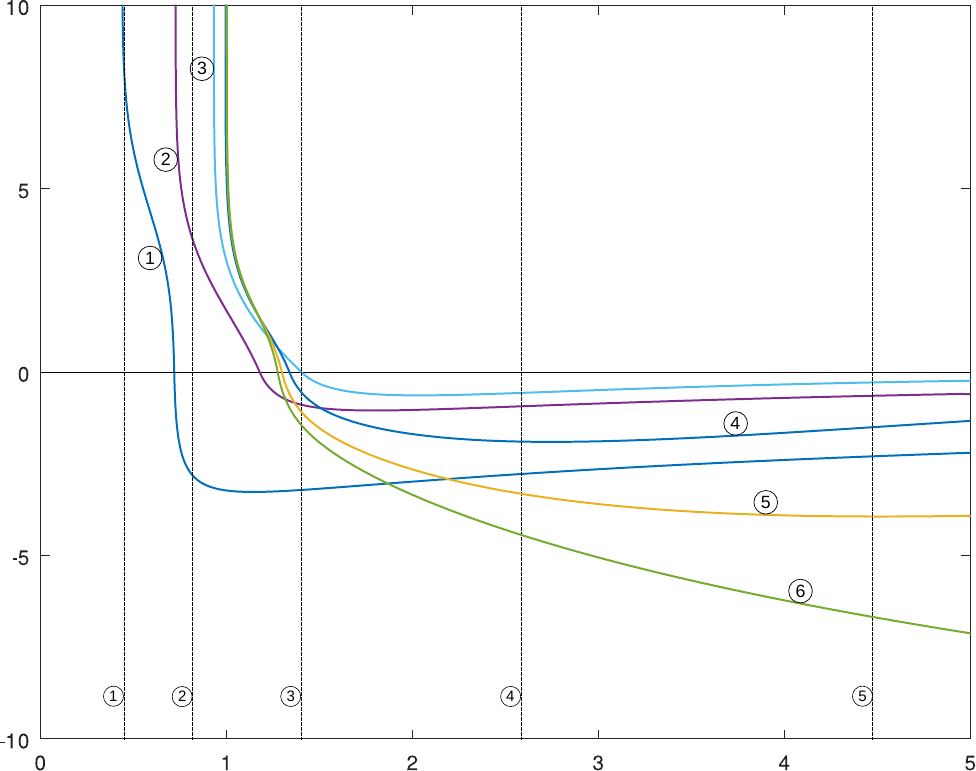}}\vspace{1mm}
 \caption{Dependence of $\mu_2$ on $d$ in a semilogarithmic scale for fixed $a=-10$, $-3$, $a_0$, $-0.3$, $-0.1$, and $0$ (plots $1,\dots,6$).}
\label{fig:logmu2_aneg}
\end{figure}

From the asymptotic formulas \eqref{eq:mu2as_d_large}, \eqref{eq:mu2as_dcrit}, \eqref{eq:mu2as_ds} it follows that $\mu_2$ tends to $+\infty$ as $d\to \dcrit+0$. If $a>0$, then $\mu_2\to-\infty$ as $d\to \ds$ and hence $\mu_2$ changes sign on the interval $(\dcrit,\ds)$. If $a<0$, then $\mu_2$ is negative for large $d$ in view of \eqref{eq:mu2as_d_large} and \eqref{eq:m1}, which implies a change of sign in $(\dcrit,+\infty)$. In both cases we denote the point where the sign change occurs by $d_0(a)$. The dependence of the zero $d_0(a)$ of $\mu_2$, obtained numerically, is shown in Fig.~\ref{fig:mu2_pos_cf}.

Figure~\ref{fig:logmu2_aneg} and~\ref{fig:logmu2_apos} present numerical results illustrating the dependence of $\mu_2$ on $d$. 
Specifically, these figures show $\operatorname{sgn}(\mu_2(d))\log(1+|\mu_2(d)|)$ for fixed values of $a$. In Fig.~\ref{fig:logmu2_aneg} the dependence is plotted for $a=-10$, $-3$, $a_0$, $-0.3$, $-0.1$, and $0$ (the curves are marked by numbers $1,2,\dots,6$, respectively). In Fig.~\ref{fig:logmu2_apos} the dependence is plotted for $a=5$, $1.5$, $0.5$, $0.25$, and $0.15$ (the curves are marked by numbers $1,2,\dots,5$, respectively). Dotted lines show the position of the corresponding values $\sqrt{2/|a|}$.

\begin{figure}[t!]
\centering\vspace{2.25mm}
 \SetLabels
 \L (-0.06*0.23) \rotatebox{90}{$\operatorname{sgn}(\mu_2(d))\log(1+|\mu_2(d)|)$}\\
 \L (0.9*-0.02) $d$\\
 \endSetLabels
 \leavevmode\AffixLabels{\includegraphics[width=90mm]{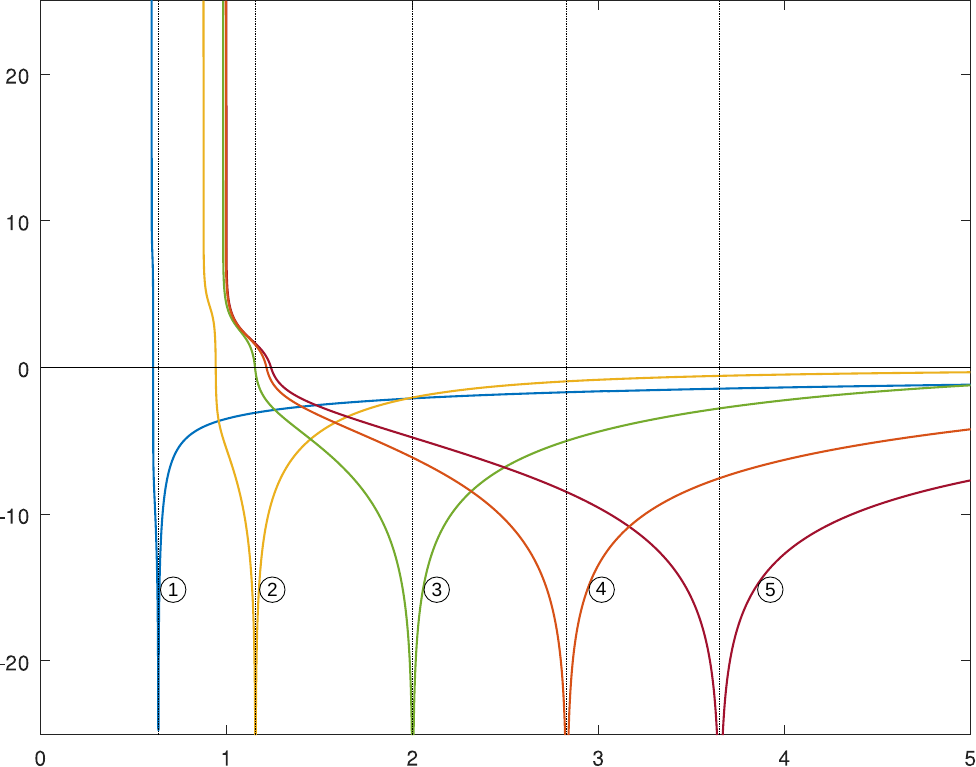}}\vspace{1mm}
 \caption{Dependence of $\mu_2$ on $d$ in a semilogarithmic scale for fixed $a=5$, $1.5$, $0.5$, $0.25$, and $0.15$ (plots $1,\dots,5$).}
\label{fig:logmu2_apos}
\end{figure}

The numerical results presented in Fig.~\ref{fig:logmu2_apos} support an important hypothesis concerning the sign of $\mu_2$. Specifically, for $a>0$, $\mu_2$ is consistently negative whenever $d\geq\ds(a)$. This threshold condition corresponds to the presence of a stagnation point in the laminar flow ($U'(y)$ changes sign as stated in Sect.~\ref{SAu13a}): at the value $d=\ds(a)$, the stagnation point first appears at the surface, whereas for $d>\ds(a)$, it is located beneath the surface.
At the same time, from Fig.~\ref{fig:mu2_pos_cf}  and Fig.~\ref{fig:logmu2_aneg}, one can note that for any $a\leq a_0\approx-1.01803$, there is an interval $d\in(\ds,d_0)$, where $\mu_2(a,d)$ is positive and, at the same time, there exists a stagnation point in the laminar flow. We denote the domain of such points $(a,d)$ by $M_+$, it is shown by color in Fig.~\ref{fig:mu2_pos_cf}.

We can also support this property analytically, by proving that $\mu_2(a,d)>0$ for some $(a,d)\in \Upsilon_-$ (see Fig.~\ref{fig:Udiff0}), where $\Upsilon_-$ is defined by
$$
 a\leq -\frac{2}{d^2}.
$$
Consider the curve 
\begin{equation}
\label{eq:acurve}
a=-\frac{4}{d^2}
\end{equation}
in $\Upsilon_-$.

Let $d$ be sufficiently small. Substitute \eqref{eq:acurve} into the dispersion relation $\Disp(\taust)=0$ along with
\begin{equation}
\taust=\taust(d)=\nmo d^{-1}+n_0+n_1 d+n_2 d^2+\dots,
\label{eq:taustsa}
\end{equation}
we find
\begin{gather*}
\nmo=\frac43\tanh \nmo,\quad\mbox{i.e.}\quad \nmo\approx 1.034021...,\\
n_0=n_1=0,\quad n_2=\frac{\nmo}{9\nmo^2-4}.
\end{gather*}
Substituting \eqref{eq:taustsa} into \eqref{eq:mu2def}, after some algebra we find that on the curve $a(d)=-4/d^2$
\[
\mu_2=M d^{-5}+\Obig(d^{-2})\quad\mbox{as}\quad d\to0,
\]
where
\[
M=\frac{729\,\nmo^6-3078\,\nmo^4+3168\,\nmo^2-512}{54\,\nmo^2}\approx 4.287466...>0.
\]
This shows that $\mu_2(a,d)$ can be positive when the laminar flow $U(y)$ on $(0,d)$ has a stagnation point inside the interval in the case $U'(0)\leq 0$, i.e.\ a counter current appears near the bottom.

\begin{figure}[t!]
\centering\vspace{1.25mm}
 \SetLabels
 \L (0.9*-0.01) $a$\\
 \L (1.01*0.87) $Y_{*,\max}$\\
 \L (-0.08*0.34) \rotatebox{90}{$Y_*(a,d_0(a))$}\\
 \endSetLabels
 \leavevmode\AffixLabels{\includegraphics[width=72mm]{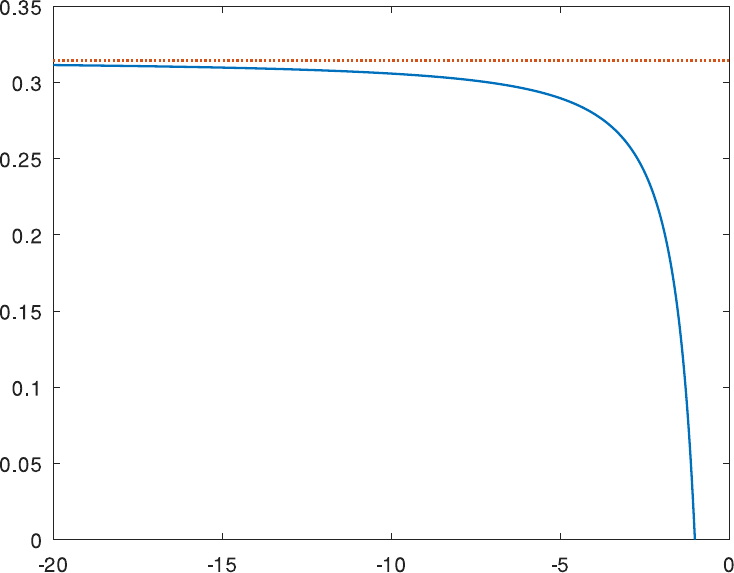}}\vspace{1mm}
 \caption{Dependence of $Y_*(a,d_0(a))$ on $a<a_0\approx-1.01803$. The dotted line shows the limit level $Y_{*,\max}\approx 0.314507...$}
\label{fig:ystar_mu2=0}
\end{figure}

We can estimate numerically how close the extremum point can be to the bottom when $(a,d)\in M_+$.  
In Sect.~\ref{SAu13a} we introduced the relative depth of the extremum $U'(y_*)=0$ for given $a$ and $d$:
\begin{equation*}
Y_*(a,d)=y_*/d=\frac{ad^2+2}{2ad^2}.
\end{equation*}
Since $Y_*(a,d)$ is monotone with respect to $d$,  its maximum for $(a,d)\in M_+$ is attained on $Y_*(a,d_0(a))$ (on the upper part of $\partial M_+$). This dependence is shown in Fig.~\ref{fig:ystar_mu2=0}. According to the computations,
\begin{equation*}
Y_{*,\max}:=\sup_{(a,d)\in M_+}Y_*(a,d)=\lim_{a\to-\infty}Y_*(a,d_0(a))\approx 0.314507...
\end{equation*}

\subsection{The first eigenvalue of the Fr\'{e}chet derivative. The formal stability}\label{SFeb7c}

According to Sect.~\ref{SectJ22a}, the first eigenvalue of the problem (\ref{Ja22a}) is
$\mu_{00}=\Disp(0)$ and the corresponding eigenfunction is
$$
u_{00}(y)=\gamm(y;0)=\frac{y}{d}.
$$
Let us find the asymptotics of the first eigenvalue and corresponding eigenfunction of the problem (\ref{eq:eigenproblem}).
We write the eigensolution in the form
\begin{equation*}
u(x,y)=\gamm(y;0)+t u_0(x,y)+\Obig(t^2).
\end{equation*}
and look for the eigenvalue 
\begin{equation}
\mu_0(t)=\Disp(0)+\mu_{01}\,t+\Obig(t^2).
\label{eq:mu0expdef}
\end{equation}

From the first equation \eqref{eq:eigenproblem}, we have
\begin{equation*}
0=t(\pd[2]{x}+\pd[2]{y})u_0(x,y)+\Obig(t^2).
\end{equation*}
Using \eqref{eq:etaexpdef},
 we can use the Taylor expansion for $u$
\begin{equation*}
u(x,\eta(x)) = u(x,d) + \zeta(x)\pd{y}u(x,y)\bigr|_{y=d} + \tfrac12\zeta(x)^2 \pd[2]{y}u(x,y)\bigr|_{y=d} +\dots
\end{equation*}
and, similarly, for its derivatives and derivatives of $\psi$ (see \eqref{eq:psix}, \eqref{eq:psiy}).
 We substitute these expressions into
\begin{equation}
\lambda^2\psi_x\psi_y u_x+(\psi_y)^2 u_y-(1+\lambda^2\psi_x\psi_{xy}+\psi_y\psi_{yy})u-\mu u=0
\label{eq:mu-cond2-mod}
\end{equation}
and equate to zero coefficients of powers of $t$. The term of order $\Obig(1)$ in \eqref{eq:mu-cond2-mod} is equal to zero due to our choice of the first term in the right-hand side of \eqref{eq:mu0expdef}.
Consider the term of order $\Obig(t)$ in \eqref{eq:mu-cond2-mod}. Using representations \eqref{eq:eta_psi0} of $\psi_0$ and $\eta_0$, the above formula for $\mu_0(t)$ and assuming
\[
u_0(x,y)=p_0\cos(\taust x)\gamm(y;\taust),
\]
we obtain, after some algebra and application of the dispersion relation,
\begin{equation*}
\cos(\taust x)\biggl\{-p_0\Disp(0)+\dUd^2\taust^2-\frac{\mddU}{\dUd}-\frac{\dUd^2}{d^2}-\frac{2}{d}\biggr\}-\mu_{01}=0,
\end{equation*}
and, thus, we find that $\mu_{01}=0$ and
\begin{equation}
p_0 = \frac{d^2\dUd^3\taust^2-\mddU d^2-\dUd^3-2d\dUd}{d^2\dUd\,\Disp(0)}.
\label{eq:p0def}
\end{equation}

Introduce now the Dirichlet--Neumann operator $\DN=\DN(\psi,\eta)$. If $h$ is a \mbox{$\Lambda_*$-}pe\-ri\-od\-ic even function from $C^{2,\alpha}(\mathbb{R})$, then $w$ is defined as the solution of the boundary value problem
\begin{equation*} 
\begin{aligned}
& (\lambda^2\partial_x^2+\partial_y^2)w=0\quad\mbox{in\!}\quad D_\eta,\\
&w=0\quad\mbox{for\!}\quad y=0,\\
&w=h\quad\mbox{for\!}\quad y=\eta(x),
\end{aligned}
\end{equation*}
and
\begin{equation*} 
\DN h=\lambda^2\psi_xw_x+\psi_yw_y\quad\mbox{on\!}\quad S_\eta.
\end{equation*}
Using the Green formula
$$
0=\int_{D_\eta} \bigl[(\lambda^2\partial_x^2+\partial_y^2)w_1\,w_2-w_1(\lambda^2\partial_x^2+\partial_y^2)w_2\bigr]\,\D x \D y=\int_{S_\eta}\bigl(\DN h_1 \, h_2-h_1\DN h_2\bigr)\frac{\D x}{\psi_y},
$$
where $w_i$ corresponds to $h_i$, one can show that the operator $\DN$ is self-adjoint with respect to the inner product
 \begin{equation*}
\langle h_1,h_2\rangle_1:=\int_{-\Lambda_*/2}^{\Lambda_*/2}h_1 h_2 \frac{\D x}{\psi_y}.
\end{equation*}
Let
$$
 v(x,y)=\psi_0(x,y)+\Obig(t)
$$
be the eigenfunction corresponding to the second eigenvalue $\widehat{\mu}_2(t)=\mu_2t^2+\Obig(t^4)$ (see \eqref{eq:vdef}).

We introduce
$$
\widehat{v}(x)=M_1^{-1/2}v(x,d)\quad\mbox{and}\quad\widehat{u}(x)=M_0^{-1/2}u(x,d),
$$
which are normalized eigenfunctions with respect to the inner product
 \begin{equation}\label{Feb13a}
 \langle h_1,h_2\rangle_2:=\int_{-\Lambda_*/2}^{\Lambda_*/2}h_1 h_2 \frac{\D x}{\psi^2_y}.
 \end{equation}
Clearly the functions $u$ and $v$ are orthogonal with respect to the inner product (\ref{Feb13a}).  One can verify that
\begin{align*}
 M_1&=\langle v(\cdot,d),v(\cdot,d)\rangle_2=\int_{-\Lambda_*/2}^{\Lambda_*/2}\bigl(\gamm(d;\taust)\cos(\taust x)\bigr)^2 \,\frac{\D x}{\psi^2_y}+\Obig(t)=\frac{\Lambda_*}{2\kappa^2}+\Obig(t),\\
 M_0&=\langle u(\cdot,d),u(\cdot,d)\rangle_2=\int_{-\Lambda_*/2}^{\Lambda_*/2} \bigl(\gamm(d;0)\bigr)^2\,\frac{\D x}{\psi^2_y}+\Obig(t)=\frac{\Lambda_*}{\kappa^2}+\Obig(t).
\end{align*}

Let us introduce the operator
$$
{\mathcal A}h=\DN h-\frac{\widehat{\rho}}{\psi_y} h,
$$
where $\widehat{\rho}$ is defined by \eqref{eq:rhohatdef}.
Representing $h$ as $h=\alpha \widehat{u}+\beta\widehat{v}+\widehat{h}$, where $\widehat{h}$ is orthogonal to $\widehat{u}$ and $\widehat{v}$ with respect to the inner product (\ref{Feb13a}), we have
\begin{equation*}
{\mathcal A}h=\alpha\widehat{\mu}_0\frac{\widehat{u}}{\psi_y}+\beta\widehat{\mu}_2\frac{\widehat{v}}{\psi_y}+{\mathcal A}\widehat{h}
\end{equation*}
and, therefore,
\begin{equation}\label{Feb8a}
\langle {\mathcal A}h,h\rangle_1=
\alpha^2\widehat{\mu}_0+\beta^2\widehat{\mu}_2+\langle {\mathcal A}\widehat{h},\widehat{h}\rangle_1.
\end{equation}
If we denote by $\widehat{\mu}_3(t)$ the third eigenvalue of the operator ${\mathcal A}$, then
\begin{equation*}
\langle {\mathcal A}\widehat{h},\widehat{h}\rangle_1\geq \widehat{\mu}_3(t)\langle\widehat{h},\widehat{h}\rangle_2.
\end{equation*}
Assuming that
\begin{equation}\label{Feb8aa}
\int_{-\Lambda_*/2}^{\Lambda_*/2}h\,\frac{\D x}{\psi_y}= 0,
\end{equation}
let us study the positivity of the right-hand side of (\ref{Feb8a}). This positivity is called the formal stability according to \cite{CSst2}. 

First, we rewrite (\ref{Feb8aa}) as
\begin{equation*}
\int_{-\Lambda_*/2}^{\Lambda_*/2}h\,\psi_y(x,d)\frac{dx}{\psi^2_y(x,d)}=0.
\end{equation*}
We recall
$$
\psi_y(x,d)=\kappa-t\kappa\gamm'(d;\taust)\cos(\tau_*x)+t^2\psi_*(x,t),
$$
where $\psi_*(x,t)$ is analytic with respect to  $x$ and $t$ for small $t$. Using that
\begin{equation*}
u(x,d)=1+tp_0\cos(\taust x)+\Obig(t^2),\quad v(x,d)=\cos(\taust x)+\Obig(t),
\end{equation*}
we write
$$
\psi_y(x,d)=\kappa u(x,d)-t\kappa\mathcal{C} v(x,d)+\Obig(t^2),\quad
\mathcal{C}=p_0+\gamm'(d;\taust).
$$
Now the relation (\ref{Feb8aa}) takes the form
\begin{equation*}
\int_{-\Lambda_*/2}^{\Lambda_*/2}(\alpha \widehat{u}+\beta \widehat{v}+\widehat{h})\bigl(u-t\,\mathcal{C}v+\Obig(t^2)\bigr)\,\frac{\D x}{\psi^2_y}=0,
\end{equation*}
which implies
\begin{equation*}
\alpha M_0^{1/2} - t \beta \mathcal{C} M_1^{1/2}=\Obig\bigl(t^2\|\widehat{h}\|_2\bigr).
\end{equation*}
Substituting $\alpha$ from the last equation in (\ref{Feb8a}) and taking into account that $M_1/M_0=\tfrac12+\Obig(t)$, we get
\begin{equation*}
\langle {\mathcal A}h,h\rangle_1=t^2\mathcal{B}\beta^2+\Obig(t^3),
\end{equation*}
where
\[
\mathcal{B}:=\frac{\mathcal{C}^2}{2}\widehat{\mu}_0+\mu_2.
\]

\begin{figure}[t!]
\centering\vspace{1.25mm}
 \SetLabels
 \L (-0.02*0.83) $d$\\
 \L (0.91*-0.02) $a$\\
 \L (0.66*0.67) $B_+$\\
 \L (0.9*0.54) $\dcrit$\\
 \L (0.865*0.92) $\ds$\\
 \L (0.775*0.92) $\ds$\\
 \L (0.62*0.825) $d_0$\\
 \L (0.71*0.08) $a_0$\\
 \L (0.855*0.08) $a_1$\\
 \endSetLabels
 \leavevmode\AffixLabels{\includegraphics[width=85mm]{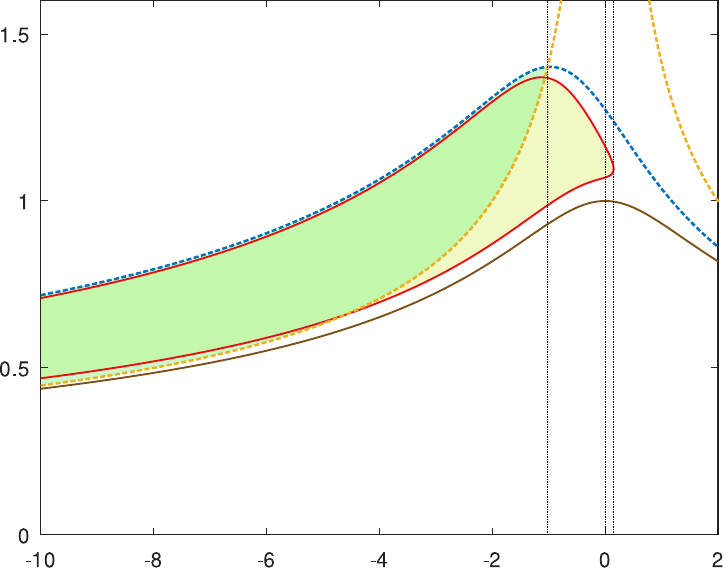}}\vspace{1mm}
 \caption{The formal stability (positiveness of $\mathcal{B}(a,d)$ takes place in the domain $B_+$ marked by color and bounded by solid line).  The domain $M_+$ is also shown in color and bounded by dashed lines ($\ds(a)$ and $d_0(a)$); see also Fig.~\ref{fig:mu2_pos_cf}.}
\label{fig:Bpos}
\end{figure}
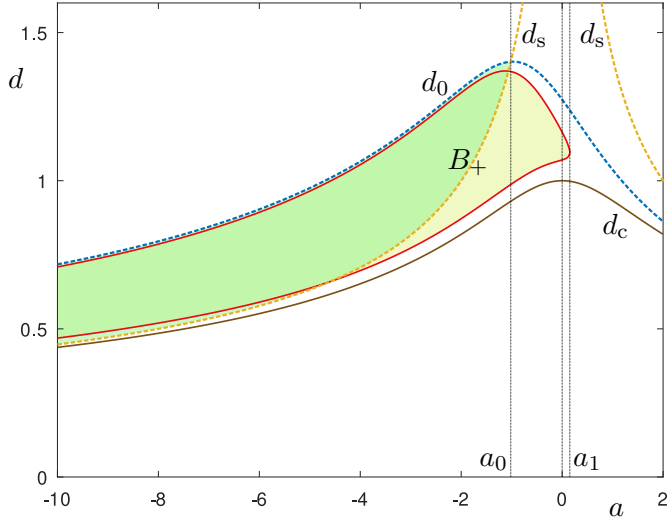

Therefore, for positivity of the form for small $t$, we must require positiveness of $\mathcal{B}$. Taking into account that $\widehat{\mu}_0=\Disp(0)$, using \eqref{eq:p0def}, \eqref{eq:taucritexpansion} and \eqref{eq:mu2as_dcrit}, after some algebra we find
the asymptotics at fixed $a$ and $d\to\dcrit+0$:
\begin{equation}
\mathcal{B}=\mathcal{B}_{-1}(d-\dcrit)^{-1}+\mathcal{B}_{0}+
\Obig(d-\dcrit),
\label{eq:Basympt}
\end{equation}
where
\begin{equation*}
\mathcal{B}_{-1}:=\frac{\dcrit^3-4}{12\,\dcrit^4},\qquad
\mathcal{B}_{0}:=\frac{13\,\dcrit^6+15\, a \dcrit^5-209 \dcrit^3-195\, a \dcrit^2+358}{30\, \dcrit^5(2-\dcrit^3-a \dcrit^2)}.
\end{equation*}
The coefficient $\mathcal{B}_{-1}(a)$  is negative and also, using \eqref{eq:dcritasympt}, we find
\begin{equation}
\label{eq:B-1asympt}
\mathcal{B}_{-1}(a)=-\frac{a^2}{12}-\frac{|a|^{1/2}}{2^{5/2}}+\Obig(a^{-1})\quad\mbox{as\!}\quad a\to\infty.
\end{equation}
In view of \eqref{eq:dcritasympt}, we have 
$$
\mathcal{B}_{0}(a)=\frac{187\,|a|^{5/2}}{15\cdot 2^{7/2}}+\frac{175\,|a|}{48}+\Obig(a^{-1/2})\quad\mbox{as\!}\quad a\to-\infty,
$$
i.e.\ the second term of \eqref{eq:Basympt} is positive and grows faster than \eqref{eq:B-1asympt} as $a\to-\infty$. We note that the asymptotics of $\mathcal{B}_{0}(a)$ for positive $a$ is essentially different:
$$
\mathcal{B}_{0}(a)=-\frac{2^{1/2}\,a^{11/2}}{15}-\frac{7\,a^4}{15}
+\Obig(a^{5/2})\quad\mbox{as\!}\quad a\to+\infty.
$$

Numerical computations confirm that the function $\mathcal{B}(a,d)$ is positive in the domain $B_+$ shown in color in Fig.~\ref{fig:Bpos}, where other notation of  Fig.~\ref{fig:mu2_pos_cf} is used. The domain $M_+$ (where both $\mu_2(a,d)>0$ and a counter current appears near the bottom) is shown in color and bounded by dashed lines $\ds(a)=\sqrt{2/|a|}$ and $d_0(a)$ such that $\mu_2(a,d_0)=0$. In Fig.~\ref{fig:Bpos} the lower curve shows $\dcrit(a)$, $a_0\approx-1.01803$ and $a_1\approx0.15196$.

We should note that the boundary of $B_+$ is fairly close to the curve $d_0(a)$. Besides, due to the presence of the first negative singular term in asymptotics \eqref{eq:Basympt}, the boundary of $B_+$ is separated from the curve $\dcrit(a)$ for any value of~$a$.

\appendix

\renewcommand{\theequation}{A.\arabic{equation}}
\setcounter{equation}{0}

\section*{Appendix A. Asymptotics of $\taust$}

We recall that $\taust$ satisfies the dispersion relation
\begin{equation}
\Disp(\taust)=\dUd^2 \taust \coth(\taust d)+\mddU\dUd-1=0,
\label{eq:dispinz}
\end{equation}
where $\dUd=U'(d)=\frac{1}{d}-\frac{ad}{2}$.

First, let us  find the asymptotics of $\taust$ for large $d$ and $a\neq0$.  We seek $\taust$ in the form
\begin{equation}
\taust= q_1d^{-1} + q_2d^{-2} + \dots
\label{eq:taudefinf}
\end{equation}
Substituting this into  \eqref{eq:dispinz} and taking in consideration only terms of order $\Obig(d)$, we have
\[
\frac{a^2(q_1-2\tanh(q_1))}{4\tanh(q_1)}=0,
\]
and, thus, $q_1\approx 1.915008...$ is the solution to the equation
\begin{equation}
q_1=2\tanh(q_1).
\label{eq:C1def}
\end{equation}
Considering terms of order $\Obig(1)$ in $\Disp(\taust)$ and using that $\tanh(q_1)= q_1/2$, we obtain
\begin{equation*}
q_2=\frac{4q_1}{a^2(q_1^2-2)}.
\end{equation*}
Continuing this procedure, we can obtain further terms in (\ref{eq:taudefinf}).

Let us now find the asymptotics of $\taust$ near $\dcrit$, defined by \eqref{Ju8b}.
Introduce the small parameter \mbox{$\epsilon=d-\dcrit$} and consider the following asymptotic anzats:
\begin{equation}
\taust=s_1\epsilon^{1/2}+s_2\epsilon+
s_3\epsilon^{3/2}+s_4\epsilon^2+\dots
\label{eq:taucritexpansion}
\end{equation}
Substituting it into the dispersion relation, we find that the term of order $\Obig(1)$ is equal to zero
in view of definition of $\dcrit$. Taking terms of the order $\Obig(\epsilon)$ in the expression for $\Disp(\taust)$, we find
\begin{equation*}
s_1=\frac{\sqrt{3}\sqrt{a^{2} \dcrit^{4}+12}}{\sqrt{\dcrit}\dcrit( 2-a \dcrit^{2} )}.
\end{equation*}
The coefficient of $\epsilon^{3/2}$ in $\Disp(\taust)$ is equal to $s_1 s_2 (a \dcrit^2-2)^2/(6\dcrit)$ and, thus, we find $s_2=0$. Taking the terms of order $O(\epsilon^2)$, we have
\begin{equation*}
s_3=\frac{
4\sqrt{3}(14\,\dcrit^6+5\,a\dcrit^5-92\,\dcrit^3-50\,a\dcrit^2+84)}
{5\,\dcrit^{5/2}(2-a\dcrit^2)^3\sqrt{4-\dcrit^3}},
\end{equation*}
where  the relation \eqref{Ju8b} is used.
Taking the terms of  order $\Obig(\epsilon^{5/2})$, we get $s_1 s_4 (a \dcrit^2-2)^2/(6\, \dcrit)=0$ and, thus, $s_4=0$. Continuing this procedure, we can find other coefficients in \eqref{eq:taucritexpansion}.


Let now $a>0$. Consider the asymptotics of $\taust$ as $d\to \ds:=\sqrt{2/a}$ (when $\Disp(\taust)\to-1$ and hence $\taust\to\infty$). Since $\coth(\ds\taust)$ approaches 1 exponentially, we
can consider approximated dispersion relation
\begin{equation*}
\dUd^2 \taust+\mddU\dUd-1=0,
\end{equation*}
where $d=\ds+\varepsilon$, and write
\[
\taust=p_{-2}\varepsilon^{-2}+p_{-1}\varepsilon^{-1}+p_0+p_1\varepsilon^{1}+\dots
\]
Then after straightforward calculations, we arrive at the following asymptotics
\begin{equation}
\taust=\frac{1}{a^2\varepsilon^2}+\frac{1+\sqrt{2}\,a^{3/2}}{\sqrt{2}\,a^{3/2}\varepsilon}+\frac{2\sqrt{2}\,a^{3/2}-1}{8a}
+\Obig\bigl(\varepsilon\bigr)\quad\mbox{as}\quad\varepsilon=d-\ds\to0.
\label{eq:taustasinf}
\end{equation}

\section*{Acknowledgments}

\noindent The study of V.\,A.~Kozlov was carried out with the financial support of the Ministry of Science and Higher Education of the Russian Federation in the framework of a scientific project under agreement No.~075-15-2025-013.
The research of O.\,V.~Motygin is funded by the Ministry of Science and Higher Education of the Russian Federation through project No.~124040800009-8.

\printbibliography

\end{document}